\documentclass[11pt,reqno]{amsart}

\usepackage{fullpage} 
\usepackage[disable]{todonotes}

\usepackage{amsmath,amssymb,amsthm}
\usepackage{color} 
\usepackage{ifthen} 
\usepackage{array} 

\usepackage{graphicx}  
\usepackage{epsfig} 
\graphicspath{{figures/}}
\usepackage{subfigure} 
\usepackage{minibox} 

\usepackage[draft]{hyperref}

\theoremstyle{plain} 
\newtheorem{theorem}{Theorem}[section]  
\newtheorem{lemma}[theorem]{Lemma}
\newtheorem{proposition}[theorem]{Proposition}
\newtheorem{corollary}[theorem]{Corollary}

\theoremstyle{definition}
\newtheorem{definition}[theorem]{Definition}

\newtheorem*{theoremBehrstock}{Theorem~\ref{thm:Behrstock}}
\newtheorem*{theoremboundedgeodesicimage}{Theorem~\ref{thm:boundedgeodesicimage}}

\let\ssection=\section
\renewcommand{\section}{\setcounter{equation}{0}\ssection}

\newcommand{\BZ}{\mathbb{Z}} 
\newcommand{\R}{\mathbb{R}} 

\newcommand{\mA}{\mathcal{A}}

\newcommand{\co}{\colon\thinspace}

\newcommand{\CC}{ \mathcal{C} }
\newcommand{\SC}{ \mathcal{S} }

\newcommand{\Snosep}{ \mathcal{S}_{nosep} }
\newcommand{\DS}{ \mathcal{DS} }
\newcommand{\DSnosep}{ \mathcal{DS}_{nosep} }

\newcommand{\Out}[1][F_n]{\mathop{\mathrm{Out}}(#1)}
\newcommand{\F}{F_n}

\newcommand{\cv}[1][n]{\mathop{\mathrm{cv}}_{#1}}
\newcommand{\CV}[1][n]{\mathop{\mathrm{CV}}_{#1}}

\newcommand{\vol}[1][t]{\mathrm{vol}_{#1}}
\newcommand{\dvol}[1][t]{\mathrm{vol}'_{#1}}

\newcommand{\core}{\mathcal{C}}

\DeclareMathOperator{\diam}{diam}

\newcommand{\into}{\hookrightarrow}

\begin{document}


\title[Submanifold Projection]{Submanifold Projection}


\author[L.~Sabalka]{Lucas Sabalka}
\address{
  Department of Mathematics and Computer Science\\
  Saint Louis University}
\email{\href{mailto:lsabalka@slu.edu}{lsabalka@slu.edu}}

\author[D.~Savchuk]{Dmytro Savchuk}
\address{
  Department of Mathematics and Statistics\\
  University of South Florida}
\email{\href{mailto:savchuk@usf.edu}{savchuk@usf.edu}}


\begin{abstract}
One of the most useful tools for studying the geometry of the mapping class group has been the subsurface projections of Masur and Minsky.  Here we propose an analogue for the study of the geometry of $\Out$ called submanifold projection.  We use the doubled handlebody $M_n = \#^n S^2 \times S^1$ as a geometric model of $F_n$, and consider essential embedded $2$-spheres in $M_n$, isotopy classes of which can be identified with free splittings of the free group.  We interpret submanifold projection in the context of the sphere complex (also known as the splitting complex).  We prove that submanifold projection satisfies a number of desirable properties, including a Behrstock inequality and a Bounded Geodesic Image theorem.  Our proof of the latter relies on a method of canonically visualizing one sphere `with respect to' another given sphere, which we call a sphere tree.  Sphere trees are related to Hatcher normal form for spheres, and coincide with an interpretation of certain slices of a Guirardel core.
\end{abstract}

\maketitle

\section{Introduction}

\todo{turn off todo, use fullpage package}

The study of the outer automorphism group $\Out$ of a free group $F = F_n$ of rank $n$ has been heavily motivated by the methods and tools from the study of a related family of groups, the mapping class groups of surfaces.  One of the most useful tools in understanding the geometry of the mapping class group is the work of Masur and Minsky \cite{masur_m:curve_complexII} on the hierarchy decomposition of Teichm\"uller geodesics, via the notion of subsurface projection and its relationship with the Harvey curve complex \cite{Harvey:curve_complex81}.  This complex carries a natural action of the mapping class group, is finite dimensional and has infinite diameter, but is not locally finite.  It is however Gromov hyperbolic \cite{masur_m:curve_complexI}, which is one key piece of the hierarchy machinery.

Because of the strong connection between $\Out$ and the mapping class group, analogous complexes have been sought for $\Out$.  Two candidate complexes carrying $Out(F_n)$ actions are quickly shaping up as the most likely analogues:  the \emph{splitting complex} and the \emph{factor complex}.

Algebraically, the \emph{splitting complex} $\SC$ of a free group $F = F_n$ of rank $n \geq 3$ is the complex whose $k$-simplices are conjugacy classes of $(k+1)$-edge free splittings of $\F$.  The \emph{factor complex} $\mathcal{F}$ of $F$ is the complex whose $k$-simplices are conjugacy classes of chains of length $k+1$ in the poset of free factors of $F$ ordered by inclusion.

Both of these complexes are finite dimensional with infinite diameter \cite{kapovich_l:analogues_of_curve_complex,behrstock_bc:intersection_numbers10} and are not locally finite.  Very recently, both have been shown to be hyperbolic as well:  hyperbolicity of the factor complex was first shown by Bestvina and Feighn \cite{bestvina_f:hyperbolicity_of_ff11} and more recently by Kapovich and Rafi \cite{kapovich_r:hyperbolicity_FF12}, while hyperbolicity of the splitting complex was first shown by Handel and Mosher \cite{handel_m:hyperbolicity_of_fs12} and more recently by Hilion and Horbez \cite{hilion_h:hyperbolicity_FS12}.

The definitions given above are algebraic in nature.  Complementing the algebraic approach to these objects and the study of $\Out$ is a topological approach based on using the doubled handlebody as a geometric model for $F_n$, dating back to the work of Whitehead in the 1930s.  Indeed, the first time the splitting complex was studied, by Hatcher \cite{hatcher:homological_stability95}, a topological definition was given.  The second proof of the hyperbolicity of the splitting complex by Hilion and Horbez utilizes this point of view.  This topological approach is a rich and interesting point of view, and one which lends itself well to intuition acquired from the mapping class group.

In this paper we approach the study of $\Out$ from this topological viewpoint.

To indicate precisely what we mean, we need some definitions and notation.  Free splittings of the free group $\F$ can be identified with isotopy classes of essential embedded $2$-spheres in the doubled handlebody $M_n := \#^n S^2 \times S^1$, analogous to the fact that $\BZ$-splittings of a surface group $\pi_1\Sigma_g$ can be identified with isotopy classes of essential embedded $1$-spheres (i.e. simple closed curves) in the surface $\Sigma_g = \#^g S^1 \times S^1$. Let $X$ be a subset of $M_n$. A \emph{sphere system} in $X$ is a finite union of disjointly embedded essential $2$-spheres in $X$ that are pairwise non-homotopic and not boundary parallel.  A sphere system is \emph{simple} in $X$ if each component in its complement is simply connected, and \emph{reduced} if its complement is simply connected (i.e. it is simple and there is only 1 component in the complement).  Note that a reduced sphere system is always simple and has exactly $n$ spheres, and that if $X$ is connected then every simple sphere system contains a reduced sphere system.  More generally, let $X \subset M_n$ be a connected component of the complement of a sphere system.  Then $X$ is homeomorphic to $M_{g,b}$, a compact $3$-manifold obtained from $M_{g}$ for some $g \leq n$ by deleting $b$ open $3$-balls with disjoint closures.  \emph{When we refer to a submanifold of $M_n$ we will always be referring to such a manifold $M_{g,b}$}. By a theorem of Laudenbach \cite{laudenbach:annals73,laudenbach:topologie_book74}, two spheres in $M_{g,b}$ are homotopic if and only if they are isotopic.

Define the \emph{sphere complex} $\SC(M_{g,b})$ to be the simplicial complex whose simplices are isotopy classes of disjoint essential embedded $2$-spheres (that is, sphere systems) in $M_{g,b}$.  This is analogous to the definition of the curve complex $\CC(\Sigma_{g,b})$, where simplices are isotopy classes of disjoint essential embedded copies $1$-spheres (that is, curve systems of simple closed curves) in $\Sigma_{g,b}$.  Via the correspondence between splittings and embedded spheres, $\SC(M_{n,0})$ is isometric to the splitting complex of $F_n$.


Roughly, subsurface projection can be defined as follows.  Let $\Sigma = \Sigma_{g,b}$ be a surface with boundary, and let $X \subset \Sigma$ denote a proper subsurface.  Let $v$ be a vertex of $\CC(\Sigma)$ and let $\gamma$ be a representative of $v$ which intersects $\partial X$ in a minimal number of components.  The projection of $v$ to $\CC(X)$ is defined to be the set of all components of $\gamma \cap X$ up to isotopy (and where we complete resulting arcs to curves in a pre-specified way).  When $X$ is a subsurface which exhausts $\Sigma$, $\CC(X)$ coincides with the intersection of the links of each component of $\partial X$ in $\CC(\Sigma)$, where for any topological spaces $A \subset B$ the subspace $A$ \emph{exhausts} $B$ if the closure of $A$ in $B$ is all of $B$.

Inspired by this topological definition, we define \emph{submanifold projection} $\pi_X(A)$ of a sphere $A$ to a submanifold $X$ of $M_{g,b}$ to be, roughly, the isotopy classes of all innermost components of $\hat A \cap X$ for $\hat A$ homeomorphic to $A$ and intersecting $\partial X$ minimally (see Section \ref{sec:projection} for the details).  Submanifold projection satisfies a number of desirable properties, including being coarsely well-defined, Lipschitz, coarsely surjective, and satisfying the following Behrstock inequality:

\begin{theoremBehrstock}
Let $X, X' \subset Y \subset M_n$ be submanifolds with boundary such that $X$ exhausts $Y$, and let $S$ be an essential embedded sphere in $Y$.  If
    $$d_{X'}([\partial X],[S]) > 3$$
then
    $$d_{X}([\partial X'],[S]) \leq 3,$$
where for a submanifold $Z$ of $Y$ and essential embedded spheres $A$ and $B$ in $Y$ we denote by $d_Z([A],[B])$ the distance between $\pi_Z([A])$ and $\pi_Z([B])$ in the disk and sphere complex corresponding to $Z$, which is quasi-isometric to $\SC(Z)$ (see Section~\ref{sec:complexes}).
\end{theoremBehrstock}

The definitions and proofs of basic facts about projection, including the Behrstock inequality are relatively straightforward, taking up only Section \ref{sec:projection}.  More complicated is the fact that this definition of projection satisfies a Bounded Geodesic Image theorem, which is the main theorem of this paper:

\begin{theoremboundedgeodesicimage}[\textbf{Bounded Geodesic Image}]  Let $S \subset Y\subset M_n$ be an essential nonseparating embedded sphere in a submanifold $Y$ of the doubled handlebody such that $Y$ exhausts $M_n$ and $\SC(Y)$ is hyperbolic.  Let $X := Y - S$.  For any geodesic segment, ray or line $\gamma$ in $\SC(Y)$ such that $\gamma$ does not contain $[S]$, the set $\pi_X(\gamma)$ has uniformly bounded diameter in $\SC(X)$.
\end{theoremboundedgeodesicimage}


A version of projection called \emph{subfactor projection} has been recently defined algebraically by Bestvina and Feighn \cite{bestvina_f:subfactor_projection12}.  Their definition of projection uses minimal invariant subtrees of associated actions on Bass-Serre trees.  They use subfactor projection to show that $\Out$ acts on a finite product of hyperbolic spaces so that every exponentially growing automorphism has positive translation length.  They also prove a version of the Bounded Geodesic Image theorem, but their restrictions on the geodesic are stronger than ours (they require that the geodesic avoids a 4-neighborhood of the vertex).  The relationship between our notion of projection and theirs is not clear.

The bulk of this paper is dedicated to setting up the proof of the Bounded Geodesic Image theorem.  To prove this theorem, we describe a way of topologically viewing slices of the Guirardel core \cite{Guirardel} as `viewing one sphere with respect to a fixed sphere system'.  The object of focus is a \emph{sphere tree}, defined in Section \ref{sec:spheretrees}.  Let $S$ denote an essential embedded sphere and let $\mA$ denote a sphere system.  We have that $S$ intersects $\mA$ in a minimal number of components if and only if $S$ satisfies a normal form condition defined by Hatcher \cite{hatcher:homological_stability95}.  It turns out that Hatcher normal form has a nice interpretation on the level of the Bass-Serre tree $T$ for the splitting corresponding to $\mA$.  This interpretation allows us to associate to $S$ a finite subtree $T_S$ of $T$ together with a finite set of points (called \emph{buds}) in $T_S$, called a \emph{sphere tree} for $S$.  The sphere $S$ may be reconstructed from $T_S$, and of course $T_S$ can be constructed from $S$, but sphere trees for a given sphere are not unique.  However, all sphere trees corresponding to spheres in Hatcher normal form homotopic to $S$ have a common core subtree, which turns out to coincide with a slice of the Guirardel core for the Bass-Serre tree for the splitting associated to $S$ and the tree $T$.

Sphere trees are thus compact combinatorial descriptions of a given sphere $S$ `from the point of view' of a given sphere system $\mA$.  Moreover, sphere trees behave nicely with respect to changing the tree $T$.  Given a \emph{folding path} $(T_t)$ in \emph{outer space} from $T$ (i.e. $T_0=T$), one may consider how $(T_t)_S$ evolves along this folding path (see Section \ref{sec:folding} for definitions).  We show that the evolution of $(T_t)_S$ along a folding path can be completely described by two fundamental rules, which we call the Bud Cancellation and Bud Exchange moves and which have topologically obvious explanations.

Our proof of the Bounded Geodesic Image theorem uses evolution of sphere trees along a folding path to find a point under projection that is common to every point along the given geodesic.  To summarize the proof, let $\gamma$ be a geodesic with endpoints $[A]$ and $[B]$, and consider the submanifold projection of $\gamma$ to a submanifold $X$ with boundary consisting of a single spherical component $S$.  By hyperbolicity, $\gamma$ is contained in a uniformly bounded neighborhood of two geodesics from $[S]$ to $[A]$ and $[B]$, respectively, which themselves can be approximated by images of folding paths from $[A]$ and $[B]$ to $[S]$ in $\SC$.  We prove that, along the image of a folding path terminating at $[S]$, every vertex has a sphere tree with respect to $S$ that contains the same specific subtree.  That common subtree  becomes a common point in the image of every such vertex under submanifold projection.

This paper is organized as follows.

In Section \ref{sec:complexes}, we introduce some complexes related to the sphere complex that make our subsequent definitions and proofs cleaner.  This includes defining a \emph{disk and sphere complex} and a \emph{nonseparating sphere complex}, the definitions of which are intuitively clear.  We prove that the complexes defined are all quasi-isometric in certain situations.

In Section \ref{sec:Hatchernormalform} we recall Hatcher normal form for viewing a sphere in the doubled handlebody so that the sphere intersects a given sphere system in a minimal number of components.

In Section \ref{sec:projection}, we define submanifold projection.  The definition and basic properties are intuitive and straightforward, and the reader interested in only these details can safely restrict attention to just this section of the paper and the preceding sections.

The remainder of the paper sets up the tools used to prove the Bounded Geodesic Image theorem.

In Section \ref{sec:spheretrees}, we define sphere trees.  We show how to construct spheres from sphere trees and sphere trees from spheres, establishing the relationship between them.  To construct sphere trees from spheres we use Hatcher normal form.  We describe the two fundamental moves (Bud Cancellation and Bud Exchange) on sphere trees.  We use these moves to define sphere trees in \emph{consolidated} form.  We choose the word `consolidated' purposefully, as we also show that these sphere trees precisely correspond with the consolidated trees constructed by Behrstock, Bestvina, and Clay \cite{behrstock_bc:intersection_numbers10}, which they prove coincide with slices of the Guirardel core.

In Section \ref{sec:folding}, we introduce two notions of quasigeodesics in curve complex analogues:  the folding paths used by Bestvina and Feighn \cite{bestvina_f:hyperbolicity_of_ff11} and the fold paths used by Handel and Mosher \cite{handel_m:hyperbolicity_of_fs12}.  These quasigeodesics are projections of paths from Culler and Vogtmann's outer space~\cite{culler_v:outer_space}, so we recall the notions related to outer space here.  Folding paths are useful for our purposes because sphere trees evolve nicely along them, by simple applications of the two moves on sphere trees.  However, folding paths are known to be quasigeodesics in the factor complex, not the sphere complex -- fold paths are quasigeodesics in the sphere complex.  These two families of paths are closely related, though:  there is a family of paths in $\SC$ where each member is both a fold path and a projection of a folding path (with full tension subgraph and all illegal turns folded at unit speed).  We call corresponding paths in the outer space \emph{terse paths}, and prove that their projections to the sphere complex form a coarsely transitive path family in this section.

The proof of the Bounded Geodesic Image theorem takes up Section \ref{sec:maintheorem}.

We end with some remarks about future directions and applications.

The definitions and results in this paper were inspired by a wonderfully inspiring discussion held at the American Institute of Mathematics in November of 2010.  We thank those present for that discussion, including but not limited to Mark Feighn, Michael Handel, Yair Minsky, and especially Karen Vogtmann, who proposed this topological approach to projection to us.  We also thank Lee Mosher, Mladen Bestvina, Patrick Reynolds, and Saul Schleimer for interesting discussions related to this material.  Most especially, we wish to thank Matt Clay, whose numerous conversations and suggestions on this material strongly shaped it.


\section{The Sphere Complex and Its Relatives}\label{sec:complexes}

Intuitively, submanifold projection should be a way of projecting a vertex in the sphere complex to the link of a fixed reference vertex.  The link of the reference vertex corresponds to all vertices that can be represented by spheres which are disjoint from a sphere $A$ representing the reference vertex -- that is, all vertices represented by spheres in the complementary submanifold $M_n - A$.

Let $S$ represent the vertex to be projected.  To find the projection, we use surgery to cut $S$ along $A$.  As such, it will be most convenient to work with disks as well as spheres, and often (to ensure that $M_n-A$ is connected) with nonseparating spheres.  The good news is that we do not lose any coarse geometric information with these restrictions, as the next definitions and proposition show.

Recall the \emph{sphere complex} of a submanifold $X$ of the doubled handlebody $M_n$ is the simplicial complex $\SC(X)$ whose $k$-simplices are isotopy classes of sphere systems with $k+1$ spheres, with faces determined by inclusion.  The \emph{nonseparating sphere complex} $\Snosep(X)$ of $X$ is the full simplicial subcomplex of $\SC(X)$ obtained by restricting to simplices with representative sphere systems consisting entirely of nonseparating spheres.  The sphere complex was defined by Hatcher \cite{hatcher:homological_stability95}, while the nonseparating sphere complex is closely related to Hatcher's complex $Y \subset \SC(X)$ (the two complexes have the same vertex set, but Hatcher only allows sphere systems whose complement is connected).

For convenience, we also define relative versions of these complexes.  The \emph{disk and sphere complex} of $X$ is the simplicial complex $\DS(X)$ whose $k$-simplices are systems of $k+1$ distinct isotopy classes of essential embedded $2$-spheres and disks rel boundary in $X$ which can all be realized disjointly.  The \emph{nonseparating disk and sphere complex} $\DSnosep(X)$ of $X$ is the full simplicial subcomplex of $\DS(X)$ obtained by restricting to simplices with representative disk-and-sphere systems whose complement in $X$ is connected.

\begin{proposition}\label{prop:Stonosep}
The inclusion map on vertices from $\Snosep(M_n)$ to $\SC(M_n)$ is a (1,2)-quasi-isometry.  The inclusion map on vertices from $\DSnosep(M_n)$ to $\DS(M_n)$ is a (1,2)-quasi-isometry.  For $X \subset M_n$ a submanifold, the inclusion map on vertices from $\SC(X)$ to $\DS(X)$ is a (1,2)-quasi-isometry.
\end{proposition}

\begin{proof}
To see that $\DS(X)$ and $\SC(X)$ are quasi-isometric, we provide a quasi-inverse to the map $\SC(X) \to \DS(X)$ induced by inclusion on the vertices.  The quasi-inverse map takes a vertex $[S]$ of $\DS(X)$ to the vertex $[cap(S)]$ of $\SC(X)$, where $cap(S)$ is defined as follows.  If $S$ is a sphere, $cap(S) := S$.  If $S$ is a disk, then $\partial S \subset \partial X$ is separating in one sphere component of $\partial X$.  Let $D$ denote either half of the separated component of $\partial X$.  Define $cap(S)$ to be the sphere $S \cup \partial S \cup D$.  As $S$ is essential, $cap(S)$ is essential.  If $S$ is embedded, then $cap(S)$ can be realized as an embedded sphere in $X$.  The two possible spheres resulting from the two possible choices for $D$ can be realized disjointly and moreover can be realized disjointly from $S$.  Thus, the two choices for $cap(S)$ represent adjacent vertices in $\SC(X)$ and represent vertices which form a simplex with $[S]$ in $\DS(X)$.  If $S_1$ and $S_2$ are two disjoint disks or spheres, then $\partial S_1$ and $\partial S_2$ can be realized disjointly, and so $cap(S_1)$ and $cap(S_2)$ can be realized disjointly.  Moreover, along a path in $\DS(X)$, choices for $D$ can be made for each disk along the path in a coherent manner, so that capping produces a path of the same length in $\SC(X)$.  It is now straightforward to see that the map from $\DS(X)$ to $\SC(X)$ induced by $[S] \mapsto [cap(S)]$ is as desired.

For nonseparating versions of these complexes on $M_n$, we again provide a quasi-inverse to the map induced by inclusion.  Suppose $S_1$, $S_2$, and $S_3$ are three essential disks or spheres in $M_n$ such that:  $S_2$ is separating, $S_1$ and $S_2$ are disjoint, and $S_2$ and $S_3$ are disjoint, but $S_1$ and $S_3$ cannot be realized disjointly.  Thus  $X - S_2$ has two components, $N_1$ and $N_2$.  Since $S_1$ and $S_3$ cannot be realized disjointly, we can assume that both are contained in $N_1$.  Since $S_2$ is essential in $M_n$, $N_2$ has nontrivial fundamental group and so $N_2$ contains a nonseparating essential sphere which is disjoint from each of $S_1,\, S_2$ and $S_3$. The result follows by performing this replacement repeatedly along any given path in $\DS(M_n)$ or $\SC(M_n)$.
\end{proof}

When referring to distances between vertices in each of these complexes, we mean their simplicial distance in the $1$-skeleton of the complex.  The distance between two sets of vertices $S_1$ and $S_2$ is the diameter of their union:   $d(S_1, S_2) := \sup_{v_1 \in S_1, v_2 \in S_2} d(v_1, v_2)$.  This is not a true distance function as the distance between a set with more than one element and itself is not 0, but it is uniformly close to a distance function for a collection of sets of uniformly bounded size, as our sets will be in all useful instances.

\section{Hatcher Normal Form}\label{sec:Hatchernormalform}

Here we recall Hatcher normal form for spheres embedded in $M = M_{g,b}$, following \cite{hatcher:homological_stability95} and \cite{hatcher_v:isoperimetric_inequalities96}.  We will use Hatcher normal form to show that submanifold projection is well-defined, and that every embedded sphere can be represented by a sphere tree.

Let $\mA$ denote a fixed sphere system in $M$.  When $\mA$ is simple, an essential embedded sphere $S \subset M$ is in \emph{Hatcher normal form} with respect to $\mA$ if $S$ meets $\mA$ transversely and every component of $S \cap \mA$ is a simple closed curve which splits $S$ into components called \emph{pieces} such that:
\begin{enumerate}
\item the boundary of each piece meets each sphere in $\mA$ in at most one component of intersection, and
\item no piece is a disk isotopic, fixing its boundary, to a subset of $\mA$.
\end{enumerate}
When $\mA$ is not simple, $S$ is in \emph{Hatcher normal form} with respect to $\mA$ if $S$ is in Hatcher normal form with respect to some simple sphere system containing $\mA$.  We extend Hatcher normal form to sphere systems by declaring a sphere system is in Hatcher normal form with respect to $\mA$ if each sphere in the system is in Hatcher normal form.

We say that $S$ intersects $\mA$ \emph{minimally} if $S$ and $\mA$ are in general position and the number of components of $S \cap \mA$ is minimal among all representatives of the isotopy class of $S$.

\begin{theorem}\cite{hatcher:homological_stability95,hatcher_v:isoperimetric_inequalities96}\label{thm:Hatchernormalform}
Every sphere system $S$ is isotopic to a sphere system in Hatcher normal form with respect to $\mA$.  The system $S$ intersects $\mA$ minimally if and only if $S$ is in Hatcher normal form with respect to $\mA$.\end{theorem}

\begin{proof}
Hatcher proves this for maximal sphere systems, and this is extended to simple sphere systems by Hatcher and Vogtmann.  Extension to non-simple sphere systems follows analogously.
\end{proof}

Note that everything above can apply to disks as well as spheres, so in fact we may talk about systems of spheres and disks embedded (rel boundary) in $M$ being in Hatcher normal form with respect to a fixed system of disks and spheres.

The above theorem shows every sphere is isotopic to some sphere in Hatcher normal form, but there are many spheres in Hatcher normal form isotopic to a given sphere.  Hatcher shows that spheres in Hatcher normal form that are isotopic are in fact \emph{equivalent}, in the following sense.

\begin{definition}
Let $S$ and $S'$ be two sphere systems in Hatcher normal form with respect to $\mA$.  We say $S$ and $S'$ are \emph{equivalent} if there exists a homotopy $h_t\co S \to M$ from $S$ to $S'$ such that $h_t$ remains transverse to $\mA$ for all $t$, and $h_t(S) \cap \mA$ varies only by isotopy in $\mA$.  In particular, the circle components of $h_t(S) \cap \mA$ stay disjoint for all $t$.
\end{definition}

\begin{theorem}\cite{hatcher:homological_stability95}
Isotopic sphere systems in Hatcher normal form are equivalent.
\end{theorem}

\begin{corollary}\label{cor:restrictingtopieces}
Isotopic sphere systems $A$ and $B$ in Hatcher normal form are isotopic via an isotopy that restricts to a homotopy in each component $C$ of $M - \mA$ that induces an isotopy on $\partial C$.  This homotopy induces a bijection between the pieces of $A$ and $B$.
\end{corollary}

\begin{proof}
Two isotopic sphere systems $A$ and $B$ in Hatcher normal form with respect to $\mA$ are equivalent, so there exists a homotopy between them that acts on their intersections with $\mA$ via isotopy.  Thus, we can modify the homotopy so that its restriction to points of $A \cap C$ always remain in $A \cap C$ for each $C$, and the homotopy induces a bijection between the pieces of $A$ and $B$.
\end{proof}

\section{Submanifold Projection} \label{sec:projection}

We are now ready to define submanifold projection.

\begin{definition}[The Projection Map $\pi$ for Spheres]
Fix submanifolds $X \subset Y \subset M_n$.  A subset of a surface is called \emph{innermost} if it is homeomorphic to a disk or a sphere.  Let $S$ denote a disk or sphere in $Y$.  The \emph{projection} $\pi_X(S)$ of $S$ onto $X$ is defined to be the collection of all components of $S \cap X$ which are innermost in $S$.
\end{definition}

For any element $D_S \in \pi_X(S)$, since $S$ is embedded in $X$, $D_S$ is embedded in $X$.  As $D_S$ is innermost in $S \cap X$, $D_S \cap \partial X = \emptyset$.  If $S$ intersects $\partial X$ \emph{minimally} -- i.e. the number of components of intersection is minimal -- then $D_S$ is essential.  There are only finitely many choices of $D_S$.  Thus, $\pi_X(S)$ is a finite set of disks or spheres in $X$.  Note $\pi_X(S)$ could be empty.

Submanifold projection for spheres induces a nice map on sphere complexes:

\begin{definition}[The Projection Map $\pi$ for Disk and Sphere Complexes]
Let $S$ denote a disk or sphere in $Y$.  The \emph{projection} $\pi_X\co \DS(Y) \to \DS(X)$ is defined to be the set of vertices
    $$\pi_X([S]) := \cup_{\hat S} [\pi_X(\hat S)],$$
where each $\hat S$ is homotopic to $S$ and intersects $\partial X$ minimally.  If $\pi_X(\hat S)$ contains no essential disk or sphere then $\pi_X([S])$ is undefined.  The restriction of the projection map $\pi_X$ to the domain $\DSnosep(Y)$ is also denoted $\pi_X\co \DSnosep(Y) \to \DS(X)$.
\end{definition}

We begin by proving that this map is coarsely well-defined.

\begin{theorem}[Coarsely Well-Defined]\label{thm:well-defined}
For any vertex $[S]$ of $\DS(Y)$ (or of $\DSnosep(Y)$) such that $\pi_X([S])$ is defined, the collection $\pi_X([S])$ has diameter 1 in $\DS(X)$.
\end{theorem}

\begin{proof}
By Theorem \ref{thm:Hatchernormalform}, the sphere $\hat S \in [S]$ used to define $\pi_X([S])$ is in Hatcher normal form.  By Corollary \ref{cor:restrictingtopieces}, the pieces of two isotopic spheres with respect to $\mA$ in $Y$ containing $\partial X$ are homotopic in the complement of $\partial X$.  By work of Laudenbach \cite{laudenbach:annals73,laudenbach:topologie_book74}, homotopic pieces in $X$ are isotopic, so the projections of two isotopic spheres coincide.  As a sphere in Hatcher normal form is embedded, all components of $\pi_X(\hat S)$ are disjoint, so $\pi_X(\hat X)$ has diameter 1.
\end{proof}

By the previous section, this map can easily be translated to $\SC(X)$ (without disks) by capping each disk in $\pi_X([S])$, making the projection have uniformly bounded diameter in $\SC(X)$.  In fact, with more careful thought, the diameter of $\pi_X([S])$ in $\SC(X)$ is at most 2.

Knowing that projection for the disk and sphere complex is coarsely well-defined, we observe some properties of projection.  For two vertices $[A]$ and $[B]$ of $\DS(Y)$, let $d_{X}([A],[B])$ be the distance function in the complex $\DS(X)$ between the projections $\pi_X([A])$ and $\pi_X([B])$.  Similarly define $d_{X}^{nosep}$ for distances in the complex $\DSnosep(X)$.

\begin{proposition}\label{prop:properties}
Assume that $X \subset Y \subset M_n$.  Let $A$ and $B$ denote disks or spheres in $Y$.  Submanifold projection satisfies the following properties:
\begin{enumerate}
\item \label{prop:nonempty}\textbf{Nonempty}:  If $X$ exhausts $Y$ then $\pi_X([A])$ is nonempty.
\item \label{prop:restrictable}\textbf{Restrictable}:  For any $Z$ such that $X \subset Z \subset Y$, $\pi_X(A) = \pi_X(\pi_Z(A))$.
\item \textbf{Coarsely Surjective}:  The $1$-neighborhood of $\pi_X(\DS(Y)))$ is all of $\DS(X)$.  If $X$ exhausts $Y$, then the $1$-neighborhood of $\pi_X(\DSnosep(Y))$ is all of $\DS(X)$.
\item \label{prop:Lipschitz}\textbf{Lipschitz}:  If there exists a geodesic $\gamma$ from $[A]$ to $[B]$ in $\DS(Y)$ such that projection to $X$ of every vertex in $\gamma$ is defined then
    $$d_{X}([A],[B]) \leq d_{Y}([A], [B]).$$
If there exists a geodesic $\gamma$ from $[A]$ to $[B]$ in $\DSnosep(Y)$ such that the projection to $X$ of every vertex in $\gamma$ is defined then
    $$d_{X}([A],[B]) \leq d^{nosep}_{Y}([A], [B]).$$
\end{enumerate}
\end{proposition}

\begin{proof}

The Nonempty property follows from the fact that if $X$ exhausts $Y$ then every component of $A-\partial X$ lives in $X$, so there is always at least one innermost component among $A-\partial X$.

The Restrictable property follows from the definition.

That $\pi$ is Coarsely Surjective follows from the fact that, given an essential disk $D$ in $X$, adjacent to $[D]$ in $\DS(X)$ are (the homotopy class of) the two essential spheres obtained by taking the union of $D$ and one of the two components of $\partial X - D$ adjacent to $\partial D$.  These two spheres are essential in both $X$ and $Y$.  If $X$ exhausts $Y$, then as components of $\partial X$ are nonseparating in $Y$, at least one of these two spheres must be nonseparating in $Y$:  the connect sum of a separating sphere and a nonseparating sphere is nonseparating.

That $\pi_X$ is coarsely Lipschitz follows from the fact that, for $[A]$ and $[B]$ adjacent vertices in $\DS(Y)$ or $\DSnosep(Y)$, there exist disjoint representatives $A$ and $B$ of the homotopy classes.  Assume without loss of generality that $A$ intersects $\partial X$ minimally.  If $B$ does not intersect $\partial X$ minimally, then it is straightforward to modify the proof of Lemma \ref{thm:well-defined} to show that there exists a homotopy which takes $B$ to have minimal number of intersections with $\partial X$ without introducing any intersections with $A$.  Thus, as $A$ and $B$ are disjoint, every component of $\pi_X(A)$ is disjoint from every component of $\pi_X(B)$, so $d_{X}([A],[B]) \leq 1$ when the projections are defined.
\end{proof}

\begin{theorem}[Behrstock Inequality]\label{thm:Behrstock}
Let $X, X' \subset Y \subset M_n$ be submanifolds with boundary such that $X$ exhausts $Y$, and let $S$ be an essential embedded sphere in $Y$.  If
    $$d_{X'}([\partial X],[S]) > 3$$
then
    $$d_{X}([\partial X'],[S]) \leq 3.$$
\end{theorem}

\begin{proof}

Since $X$ exhausts $Y$, $\pi_X([S])$ and $\pi_X([\partial X'])$ are defined.  By possibly applying a homotopy to $X \subset Y$ which is trivial outside of a small neighborhood of $\partial X'$, we may assume without loss of generality that there are no triple intersection points between $\partial X$, $\partial X'$, and $S$.  Assume that $d_{X'}([\partial X],[S]) > 3$.

\begin{figure}[!h]
\input{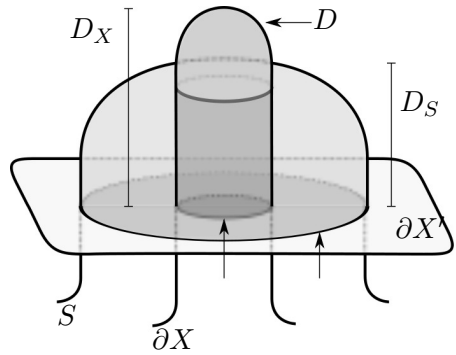}
\label{fig:Behrstock}
\caption{Labels from the proof of Theorem \ref{thm:Behrstock}.}
\end{figure}\todo{redraw}

Fix essential innermost components $D_{X} \in \pi_{X'}(\partial X)$, and $D_{S} \in \pi_{X'}(S)$.  Because by assumption $d_{X'}([\partial X], [S]) > 3$, we have that $D_{X}$ and $D_{S}$ cannot be realized disjointly rel $\partial X'$, for otherwise
    $$d_{X'}([\partial X], [S])
        \leq \diam \pi_{X'}[\partial X] + d_{X'}([D_{X}], [D_{S}]) + \diam \pi_{X'}[S]
        \leq 1 + 1 + 1  = 3.$$
As $D_{X}$ and $D_{S}$ cannot be realized disjointly, it follows that they intersect essentially relative to $\partial X$.  Since $X$ exhausts $Y$, it follows that there is an innermost component $D$ of $S - \partial X$ on $S$ which is contained in $D_{S}$.  As $D \subset D_{S}$, $D$ is disjoint from $\partial X'$.  Hence,
    $$d_{X}([\partial X'],[S])
        \leq \diam \pi_X[\partial X'] + d_{X}([\partial X'], [D]) + \diam \pi_X[S] \leq 1 + 1 + 1 = 3.$$
\end{proof}

Notice that all of the properties of projection discussed here are easiest to work with when we assume all submanifolds exhaust and we restrict our attention to the nonseparating disk and sphere complexes.

\section{Sphere Trees}\label{sec:spheretrees}

In this section, we begin by defining the notion of a sphere tree.  We show how a sphere tree is a combinatorial representation for viewing one sphere `from the viewpoint of' a given sphere system.  We then introduce Hatcher normal form for one sphere with respect to a given sphere system.  We show that every sphere can be represented by a sphere tree, with the representation related to the Hatcher normal form for the sphere.  We proceed by describing two ways of modifying sphere trees corresponding to isotopies of spheres, called the Bud Exchange move and the Bud Cancellation move.  We finally show how to use these moves to simplify a given sphere tree to represent a sphere that is in Hatcher normal form.  These moves will be used to discuss how to evolve a sphere tree along a folding path in outer space in the next section.

\subsection{Spheres From Sphere Trees}\label{sec:spheresfromspheretrees}

Fix a simple sphere system $\mA$ in $M_n$, and let $\Gamma$ denote the dual graph of $\mA$ in $M_n$, so $\Gamma$ is the marked metric graph in outer space representing the splitting corresponding to $\mA$.  Let $T$ denote the universal cover of $\Gamma$ together with the associated action of $F_n$, so in particular midpoints of edges of $T$ correspond to lifts of spheres in $\mA$ contained in the universal cover $\tilde M_n$ of $M_n$.

An \emph{(unconsolidated) sphere tree} with respect to $\mA$ is a finite subtree $T_S$ of $T$ together with a finite set of non-vertex marked points in $T_S$ called \emph{buds}, where we insist that each endpoint of $T_S$ is a bud (and hence each endpoint of $T_S$ is not a vertex).  The connected components of the complement of the buds in $T_S$ are called \emph{twigs}.  See Figure \ref{fig:unconsolidatedspheretree}.  We often identify the sphere tree with the underlying set $T_S$, but keep in mind that a sphere tree always has an associated set of buds.

\begin{figure}[!h]
\includegraphics[height=3in]{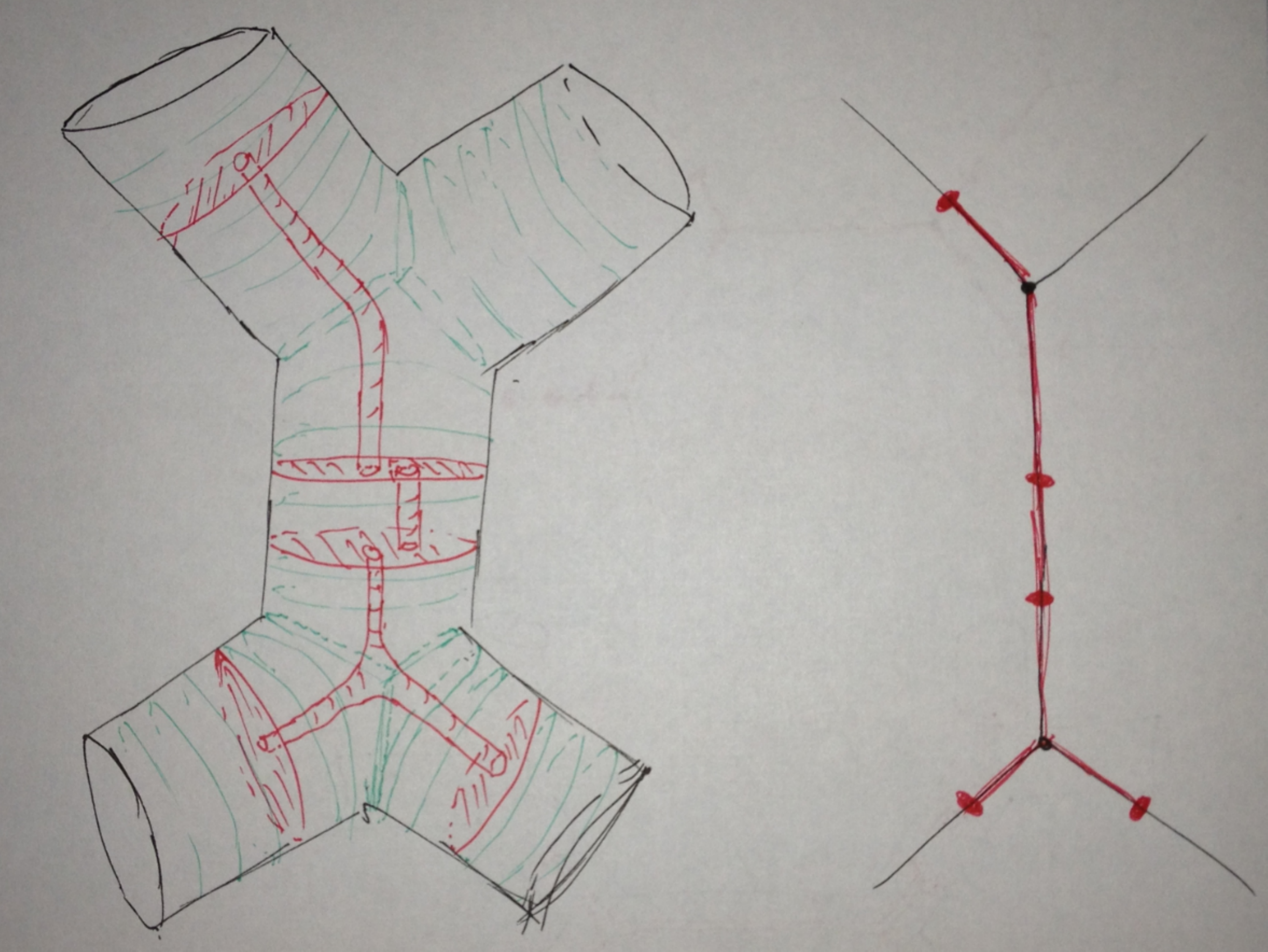}
\caption{On the right is a sphere tree (unconsolidated).  The tree $T$ is in black.  The buds of the sphere tree are red disks.  The twigs are the red paths connecting the buds.  This sphere tree has 5 buds and 3 twigs.  On the left is the associated sphere.  Its buds are spheres (shown here as disks, but each is glued to another disk in the other copy of the universal cover of the doubled handlebody (not shown) along its boundary).  Its twigs are embedded surfaces of genus 0 with 2 or 3 boundary components.  The singular fibration of $M_n$ is shown in green. \label{fig:unconsolidatedspheretree}}
\end{figure}\todo{redraw}

Because endpoints of $T_S$ are buds it follows that the finite subtree $T_S$ is equal to the convex hull of its buds.  Thus, we may define a sphere tree by simply specifying its set of buds.

In this section when we refer to sphere trees we are referring to unconsolidated sphere trees.  In future sections when we refer to sphere trees we will be referring to consolidated sphere trees, which will be defined at the end of this section.

We claim that a sphere tree in fact represents a sphere `from the viewpoint of' the sphere system $\mA$.  To support this claim we describe a construction which takes a sphere tree $T_S$ and produces a sphere $S \subset M_n$ (which will not necessarily be embedded).  See Figure \ref{fig:unconsolidatedspheretree}.  The universal cover $\tilde M_n$ can be obtained as a singular fibration over $T$ by associating to each non-vertex point of $T$ a 2-sphere and associating to each vertex of degree $k$ of $T$ the result of gluing $k$ disks along their boundaries.  This is equivalent to doubling a regular neighborhood of $T$, and gives us a fiber bundle projection from $\tilde M_n$ to $T$.  In $\tilde M_n$, start with the set of fibers over the buds of $T_S$.  These fibers each represent spheres in $\tilde M_n$ parallel to a lift of some sphere from the sphere system $\mA$.  We combine the spheres over all of the buds of $T_S$ into a single sphere in $\tilde M_n$ by taking a connected sum via tubes that follow the twigs of $T_S$.   Every twig introduces a tendril-like system of tubes connecting the spheres corresponding to the adjacent buds.  That system of tubes should be viewed as the boundary of a small neighborhood of the twig under the embedding $T_S \into \tilde M_n$.  The result is a sphere $\tilde S \subset \tilde M_n$ whose projection $S \subset M_n$ is called the \emph{sphere associated to $T_S$} (the subscript $S$ refers to this sphere, as we will usually begin with a sphere and associate to it a tree rather than vice versa).  We abuse notation and call the portions of $S$ coming from buds and twigs also by \emph{buds} and \emph{twigs}, respectively.

We now know that sphere trees represent spheres, but can any embedded sphere be represented by a sphere tree?  To answer this, we turn to Hatcher normal form.

\subsection{Sphere Trees From Spheres}\label{sec:spheretreesfromspheres}

We use Hatcher normal form we can show how to construct a sphere tree representing any given embedded sphere $S$ with respect to a given simple sphere system $\mA$.

Up to isotopy we may assume that $S$ is in Hatcher normal form with respect to $\mA$.  We choose a desired form for representatives of each isotopy class of piece of $S$.  Since $\mA$ is simple, each connected component of $M_n - \mA$ is homeomorphic to $M_{0,k}$ for some $k \geq 3$.\footnote{When $\mA$ is maximal it is automatically simple and each complementary component is $M_{0,3}$, which is homeomorphic to solid pair of pants $H_{0,3}$ doubled via the identity map along its boundary pair of pants.  Thus, when $\mA$ is maximal this decomposition of $M_n$ should be thought of as a pants decomposition.  Viewing simplices of the sphere complex as analogous to simplices in the curve complex, which are pants decompositions of the corresponding surface, should provide a point of reference for mapping class group theorists.}  Because $M_{0,k}$ has trivial fundamental group each piece $P$ of $S$ is separating, and there are only finitely many possible isotopy classes of pieces.  The isotopy class of $P$ is determined by which boundary components of $M_{0,k}$ it intersects and by the partition of components of $\partial M_{0,k} - \partial P$ it induces.

Our desired form will be defined in terms of a singular foliation of $M_{0,k}$ where all but a single leaf is a 2-sphere parallel to a boundary component of $M_{0,k}$.  The singular leaf is a union of $k$ disks glued identified along their boundary.  This is the foliation by fibers of $M_{0,k}$ induced by the fibration over the graph $Y_k$ with $1$ vertex of degree $k$ and $k$ vertices of degree 1.  Choose an embedding of $Y_k$ into $M_{0,k}$ that is transverse to the foliation, and so that the various $Y_k$ obtained in other components of $M_n - \mA$ all glue up to form the graph $\Gamma$.

Because $S$ is in Hatcher normal form, $P$ intersects each boundary component $b$ of $M_{0,k}$ at most once.  If $P$ does not intersect $b$ then we must specify which side of $P$ the boundary component $b$ is on.  If $P$ does intersect $b$ then the two components of $b - P$ are on opposite sides of $P$, and it only remains to specify which component of $b-P$ is on which side of $P$.  Thus there are 4 possible specifications for the position of $P$ with respect to $b$.  For each boundary component we take these specifications and choose a representative of the isotopy class of $P$ as in Figure \ref{fig:treeform}, so that our representative consists of boundary-parallel spheres (i.e. leaves of the foliation) called \emph{buds} of $S$ connected together via tendril-like tubes called \emph{twigs} of $S$ that stay `near' the graph $Y_k$.  Because $M_{0,k}$ is a 3-sphere with finitely many 3-balls removed and each piece is an embedded genus 0 orientable surface with boundary, these specifications uniquely determine the isotopy class of $P$.  See Figure \ref{fig:treeform}.

\begin{figure}[!h]
\includegraphics[height=2in]{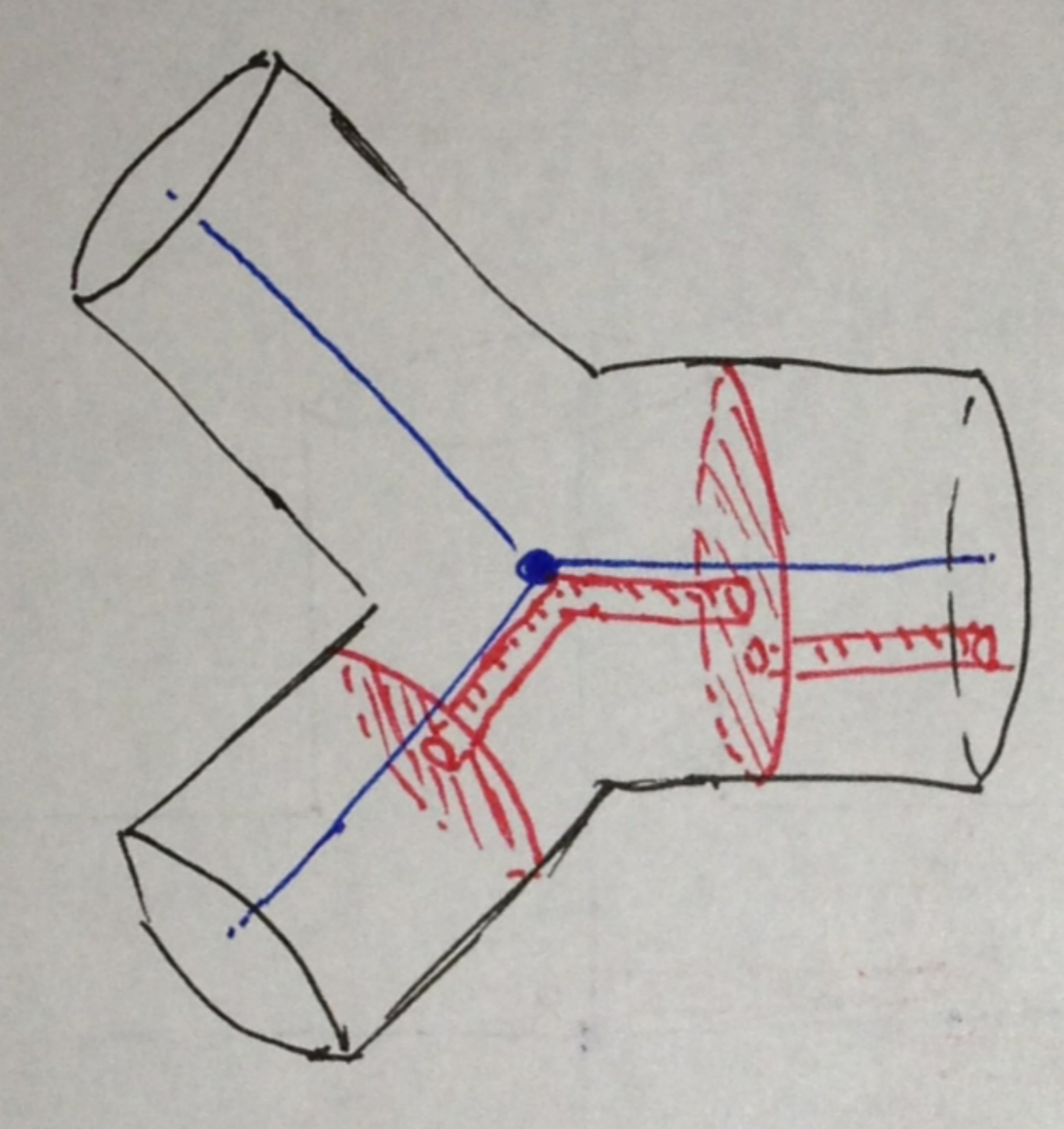}
\includegraphics[height=2.2in]{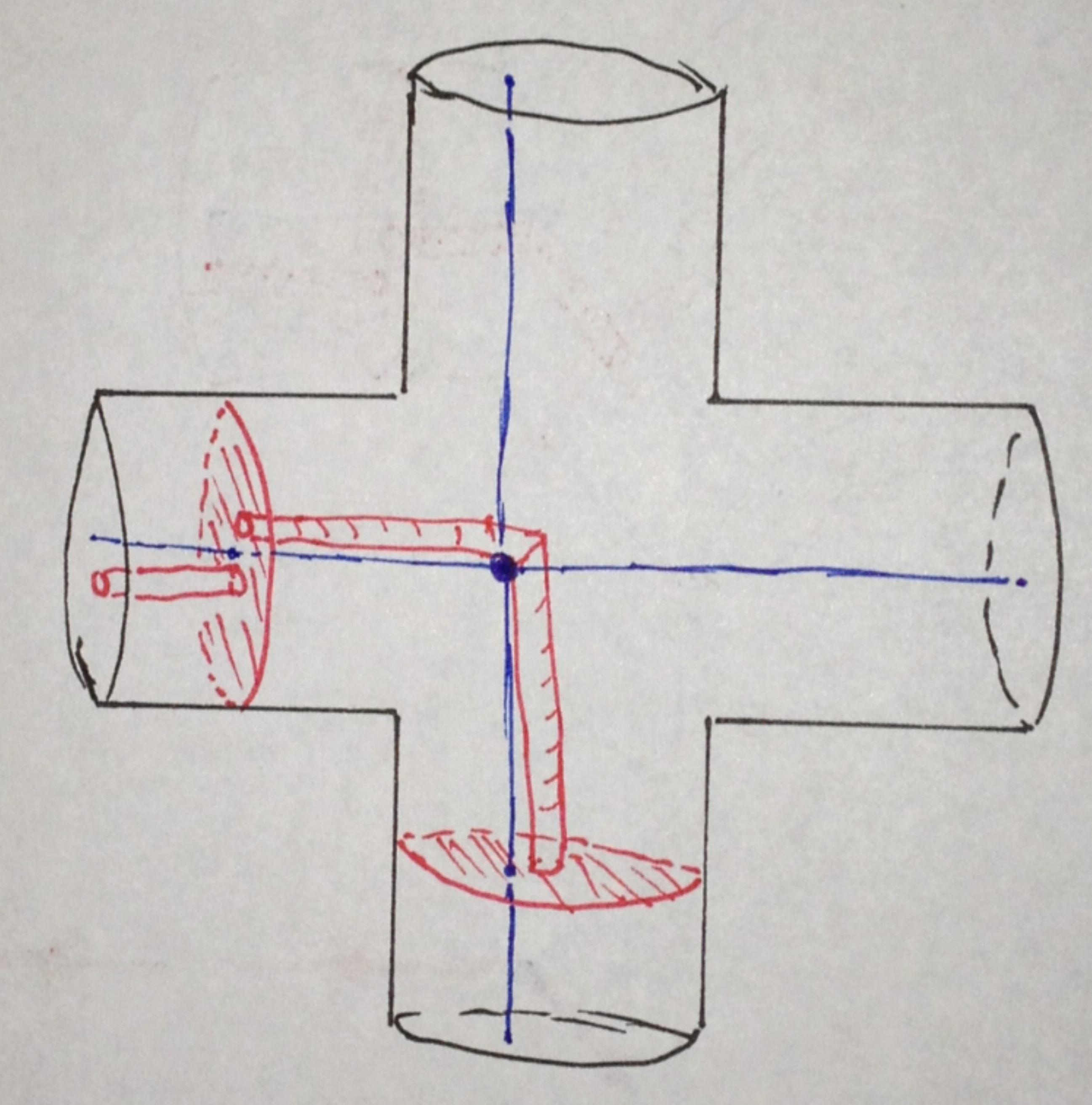}\\
\includegraphics[height=1.7in]{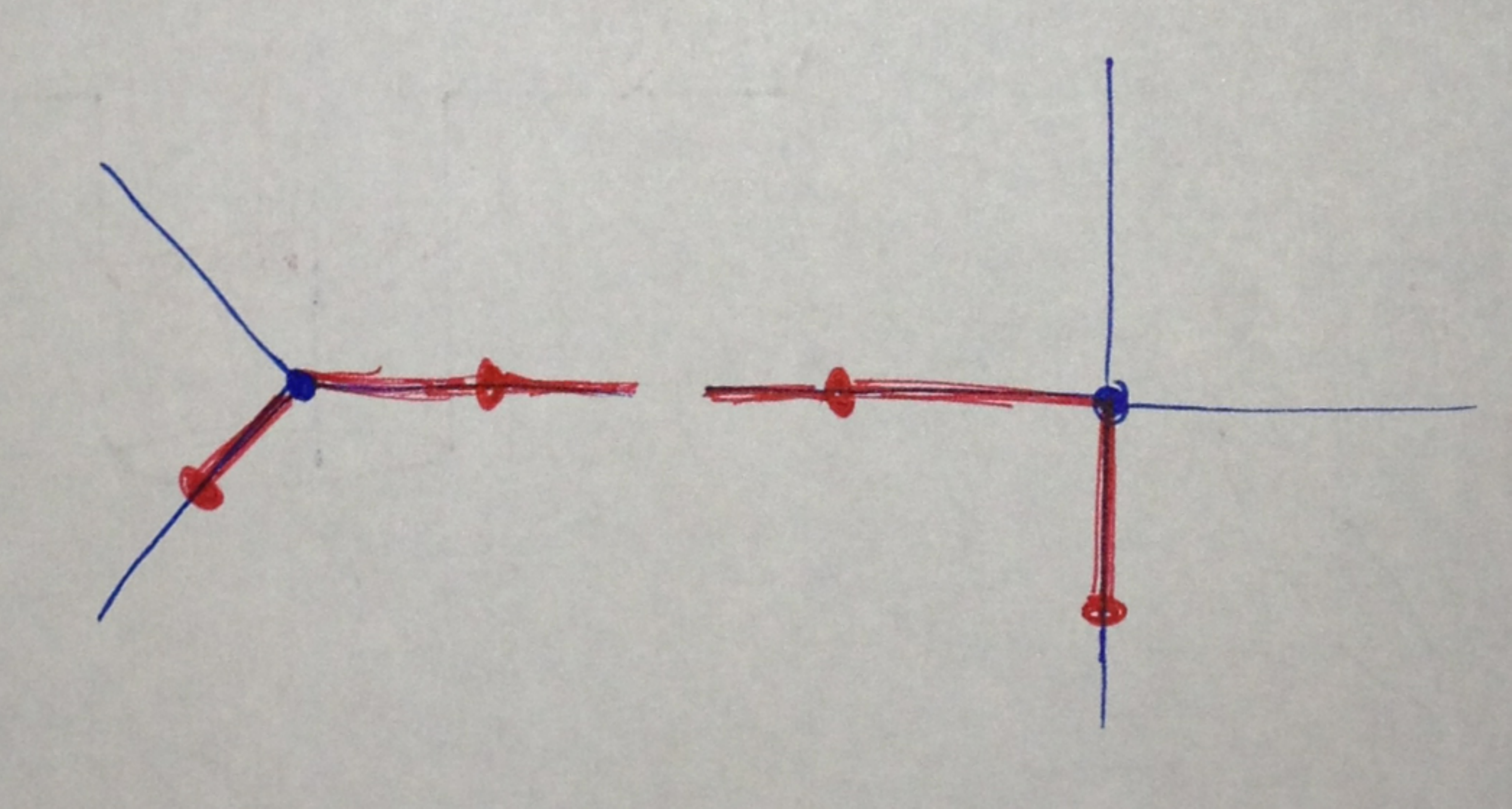}

\label{fig:treeform}
\caption{This figure shows how to choose the representative forms of each piece of the sphere $S$ in each complementary component of $\mA$, and then assemble these pieces into a sphere in tree form and produce the associated sphere tree.}
\end{figure}\todo{redraw}

That the desired forms for all of the pieces of the embedded sphere $S$ may be chosen so that $S$ is still embedded is not hard to see (apply isotopies of $M_{0,k}$ that fix each piece of $S$ in turn, making sure to not mess up previously fixed pieces or introduce any self-intersections).  If $S$ is in this desired form we say $S$ is in \emph{tree form} with respect to $\mA$.

With these desired forms for pieces of $S$, constructing a sphere tree that represents $S$ is now straightforward:  lift $S$ to $\tilde M_n$, obtain a finite subtree $T_S$ of $\tilde \Gamma = T$ by projecting $\tilde S$ to $T$ by collapsing each leaf of the singular foliation above to a point, and record as buds of the sphere tree all portions of $S$ designated as buds above.  That $S$ is the sphere associated to $T_S$ is clear by construction.  We call $T_S$ the \emph{sphere tree associated to $S$}.

\subsection{Moves on Sphere Trees}\label{sec:exchange}

Now that we have established the correspondence between sphere trees and spheres it is worth considering when two sphere trees represent isotopic spheres.  In this section we introduce two fundamental moves on sphere trees, both of which are induced by isotopy on the level of spheres.  In fact, we will see later (via evolving sphere trees along folding paths) that these two moves suffice to create the sphere tree of any essential embedded sphere with respect to any sphere system.

\subsubsection{Bud Exchange}  Let $T_S$ be a sphere tree, and let $v$ be a vertex of $T$ (that is not necessarily a vertex of $T_S$).  A bud $b$ of $T_S$ is \emph{adjacent} to $v$ if $b$ is on an edge $e_b$ incident to $v$ and there are no other buds on $e_b$ closer to $v$ than $b$.  Let $B$ denote a set of buds of $T_S$ adjacent to $v$.  Let $B^c$ denote a set of points, one per each edge adjacent to $v$ that does not contain a point of $B$, such that no bud of $T_S$ is between any point of $B^c$ and $v$.  We think of $B^c$ as a complementary set of buds for $B$.  Define a sphere tree $T'_S$ as the convex hull of the set of buds obtained from the set of buds of $T$ by replacing $B$ with $B^c$.  Bud Exchange is the result of exchanging $T_S$ for $T'_S$.  In short:\\\\
\textbf{Bud Exchange}:  \textit{At any vertex $v$ of $T$ any set of buds adjacent to $v$ can be exchanged for buds on the complementary set of edges adjacent to $v$}.\\

\begin{figure}[h]
\begin{center}
\begin{tabular}{m{80pt}m{15pt}m{80pt}}
\includegraphics[width=80pt]{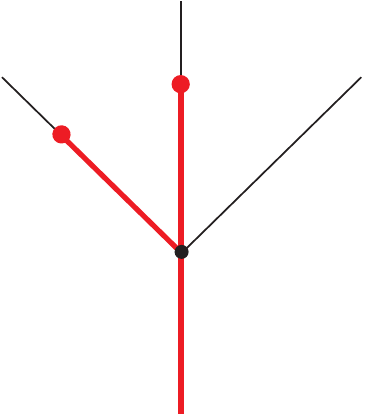}&$\longrightarrow$&
\includegraphics[width=80pt]{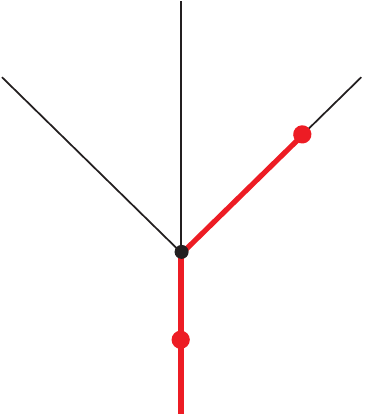}
\end{tabular}
\caption{The Bud Exchange move on sphere trees.  The central black point is the vertex $v$, while the other (red) points are the exchanged buds.}
\end{center}
\end{figure}

A set of points is \emph{innermost} in $T_S$ with respect to $v$ if no orbit of any bud of $T_S$ lies strictly between any one of the points and $v$.  A Bud Exchange move is \emph{innermost} if $v$ is adjacent to buds of both $T_S$ and $T_S'$ and the set $B \cup B^c$ of exchanged buds is is innermost.

\begin{lemma}[Bud Exchange]\label{lem:budexchange}
Two spheres associated to sphere trees which differ by a Bud Exchange move are homotopic.  If one of the spheres is embedded and the Bud Exchange move is innermost then the homotopy may be chosen to be an isotopy and the other sphere is also embedded.
\end{lemma}

\begin{figure}[h]
\begin{center}
\begin{tabular}{m{10pt}m{150pt}m{10pt}m{150pt}}
&\fbox{\includegraphics[width=150pt]{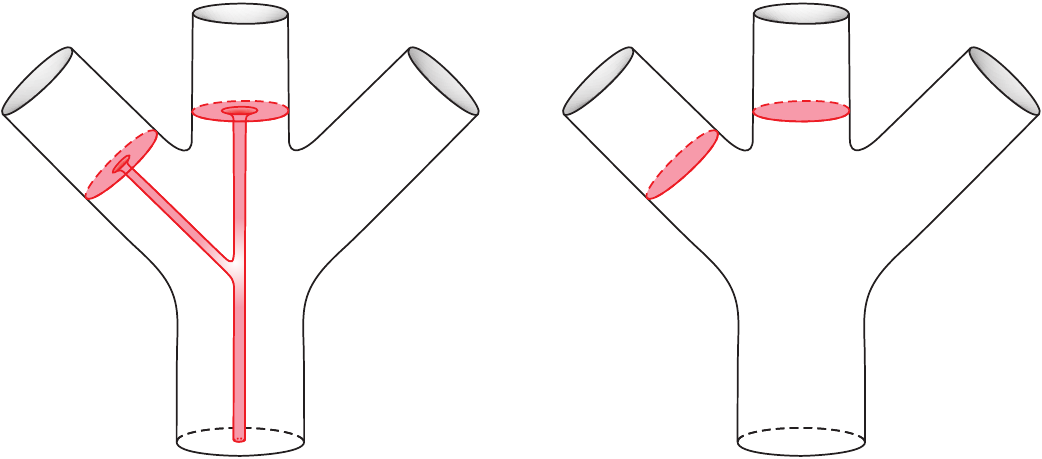}}&$\longrightarrow$&
\fbox{\includegraphics[width=150pt]{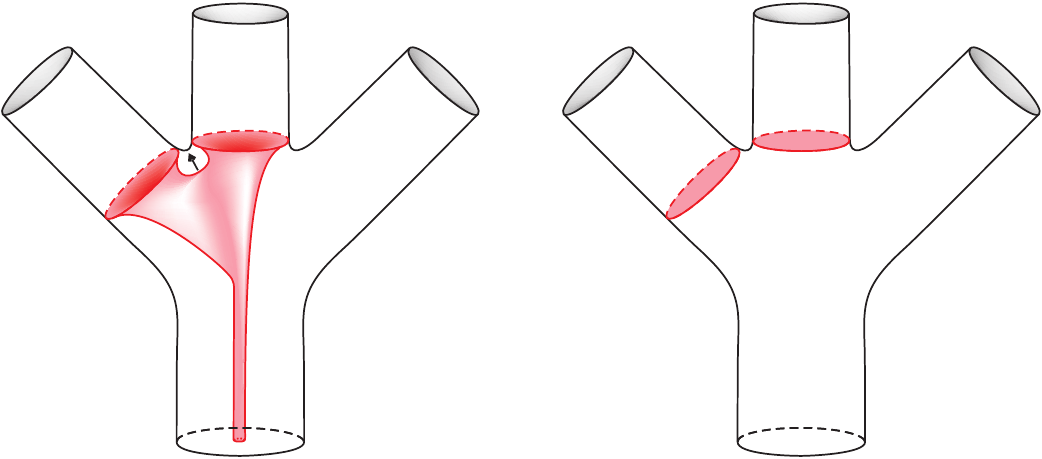}}\\\vspace{0.4cm}\\
$\longrightarrow$&\fbox{\includegraphics[width=150pt]{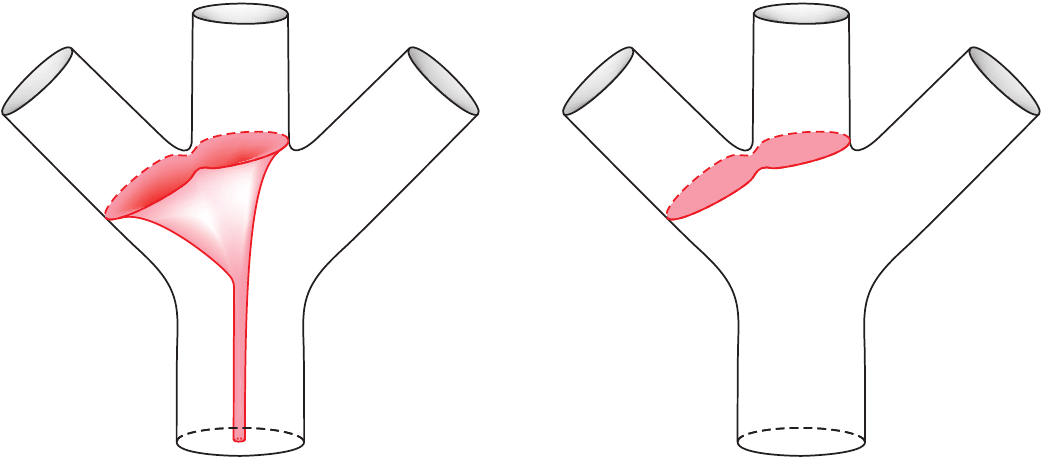}}&$\longrightarrow$&
\fbox{\includegraphics[width=150pt]{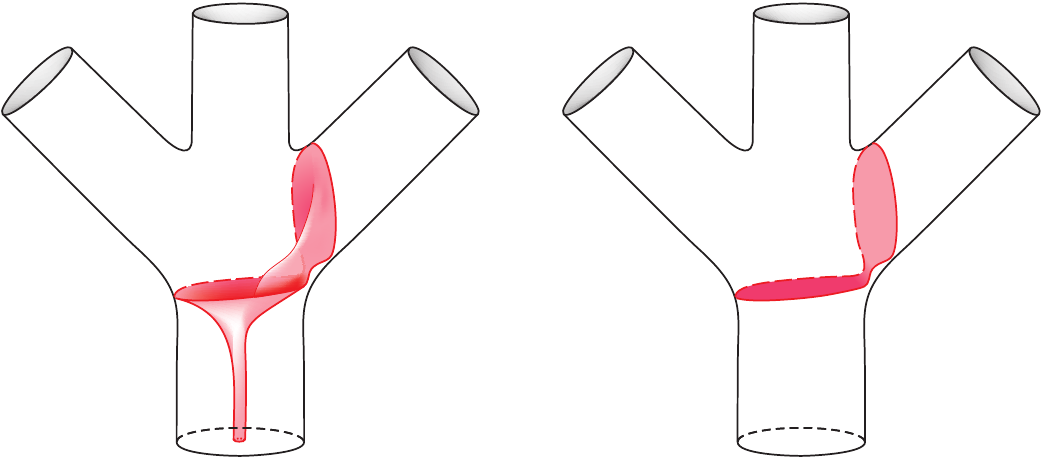}}\\\vspace{0.4cm}\\
$\longrightarrow$&\fbox{\includegraphics[width=150pt]{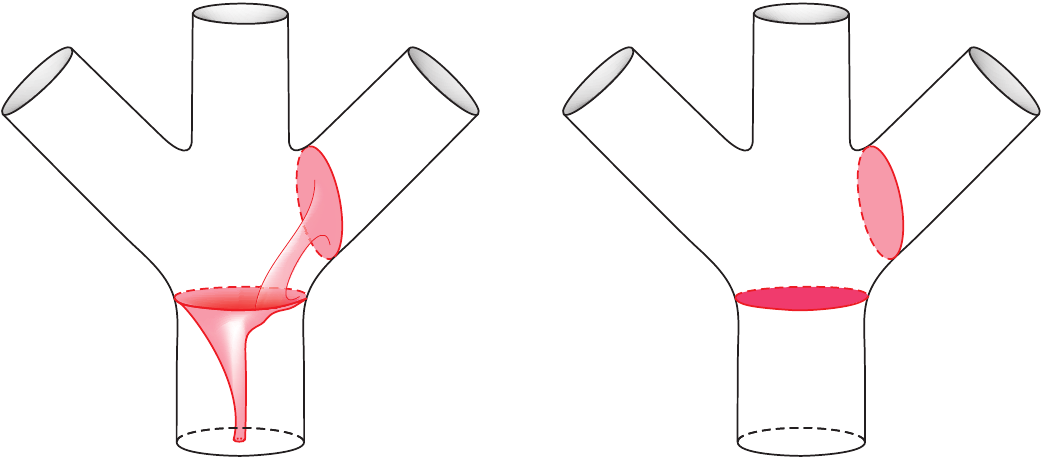}}&$\longrightarrow$&
\fbox{\includegraphics[width=150pt]{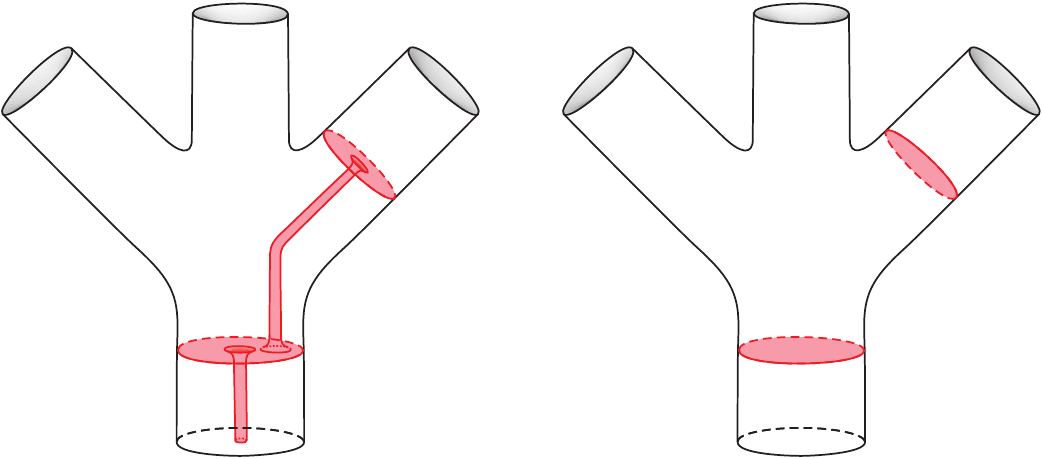}}\\
\end{tabular}
\caption{The Bud Exchange move on spheres.  Each box depicts two identical handlebodies, which are glued together via the identity map on their boundary to form the doubled handlebody.  Via this identification, the embedded surfaces depicted glue up in each case to become a disk.  The Bud Exchange move exchanges the first box for the last; the pictures show how to interpolate with an isotopy.}
\end{center}
\end{figure}

Notice that the Bud Exchange move can be formulated for spheres in tree form:  two spheres in sphere tree form that differ by exchanging a set of buds for a complementary set of buds about a vertex $v$ in the universal cover $T$ of $\Gamma$ are homotopic, and are isotopic if no other buds lie between the given buds and the image of $v$ in $M_n$.

\begin{proof}
Let $S$ denote a sphere with sphere tree $T_S \subset T$ and let $\mA$ denote the sphere system in $M_n$ corresponding to midpoints of edges of $T$.  Consider the sphere $S'$ corresponding to sphere tree $T_S'$, where $T_S$ and $T_S'$ differ by exchanging some set of buds $B$ of $T_S$ for a complementary set of buds $B^c$ of $T'_S$.

First consider the case when the only buds of $T_S$ are the buds $B$, so $T_S$ contains only the vertex $v$ of $T$.  In this case, the sphere $S$ associated to $T$ can be realized disjointly from the sphere system $\mA$, living in the component $C$ of $M_n - \mA$ containing the projection of $v$.  This component $C$ is a 3-sphere with $k$ 3-balls deleted, where $k$ is the degree of $v$.  The buds $B$ of $T_S$ correspond to boundary-parallel spheres in $C$. The sphere $S$ is the connected sum of these spheres, and is separating in $C$.  A separating sphere in $C$ is determined up to isotopy by the partition it induces on the boundary components of $C$.  In particular, the separating sphere $S'$ constructed as the connected sum of the boundary-parallel spheres corresponding to buds $B^c$ induces the same partition on boundary components.  Thus, $S$ and $S'$ are isotopic, and $T_S$ and $T'_S$ represent isotopic spheres.

Now consider the case when $T_S$ contains $v$ as well as other vertices.  In this case, let $S_v$ denote the sphere associated to the sphere tree whose set of buds is precisely $B$ (and whose underlying tree is the convex hull of $B$), and let $S'_v$ denote the sphere associated to the sphere tree with bud set $B^c$.  The sphere $S$ associated to $T_S$ is constructed as in Section \ref{sec:spheresfromspheretrees} to be the connected sum of $S_v$ with a number of other spheres corresponding to the remaining buds of $T_S$.  The isotopy between $S_v$ and $S'_v$ constructed in the previous paragraph induces a homotopy between $S$ and the sphere $S'$ associated to the sphere tree $T_S'$.  This homotopy is an isotopy if $S$ is embedded and no part of $S$ is `between' the spheres associated to $B$ and the spheres associated to $B^c$ in the component $C$ of $M_n - \mA$ containing $v$ -- that is, if $B \cup B'$ is innermost.

Finally, consider the case when $T_S$ does not contain $v$.  In this case, $B$ is at most one point.  If $B$ is nonempty then the previous paragraph still applies.  If $B$ is empty then $B^c$ consists of a bud in every direction from $v$.  The sphere $S_v$ from the previous paragraph is empty, while the sphere $S'_v$ is null-isotopic.  Let $b$ denote the bud of $T_S$ that is closest to $v$.  The homotopy from the sphere $S$ associated to $T_S$ and the sphere $S'$ associated to $T'_S$ begins by `pushing' a small disk from the bud of $S$ associated to $b$ towards $v$, forming a twig connecting the component $C$ of $M_n - \mA$ containing $v$ to $b$ by closely following the path between $b$ and $v$ in $T$.  This twig will intersect $C$ in a boundary-parallel disk.  The homotopy then proceeds via a homotopy between this disk and the null-homotopic disk consisting of a tube connecting the appropriate boundary component and $S'_v$.  As before, this homotopy is an isotopy if no part of $S$ is 'between' $b$ and $v$.
\end{proof}

Note that an exchange move is reversible by an exchange move, and if $S$ is in Hatcher normal form with respect to $\mA$ then $S'$ is too.

\subsubsection{Bud Cancellation}
Let $T_S$ be a sphere tree with an edge $e$ and two distinct buds $b_1$ and $b_2$ on $e$.  Let $T_S'$ denote the sphere tree obtained by deleting $b_1$ and $b_2$ from $T_S$, by removing $b_1$ and $b_2$ from the set of buds and gluing together the twigs adjacent to $b_1$ and $b_2$.  If either $b_1$ or $b_2$ is an endpoint of $T_S$, we also delete this consolidated twig to maintain that the endpoints of $T'_S$ are buds.  In short:\\\\
\textbf{Bud Cancellation}:  \textit{Two buds on the same edge cancel}.\\

\todo{insert bud cancellation picture for sphere trees}

The two buds $b_1$ and $b_2$ are \emph{innermost} in $T_S$ with respect to each other if no orbit of any other bud of $T_S$ lies between them.

\begin{lemma}[Bud Cancellation]\label{lem:budcancellation}
The sphere trees $T_S$ and $T'_S$ (obtained by cancelling two buds on the same edge from $T_S$) are associated to homotopic spheres.  If $S$ is embedded and $b_1$ and $b_2$ are innermost with respect to each other then the homotopy can be chosen to be an isotopy.
\end{lemma}

\begin{figure}[h]
\begin{center}
\begin{tabular}{m{10pt}m{150pt}m{10pt}m{150pt}}
&
\fbox{\includegraphics[width=75pt]{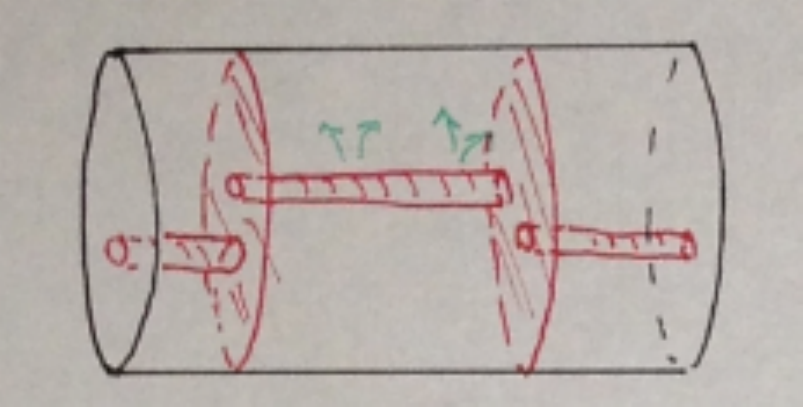}~~\includegraphics[width=75pt]{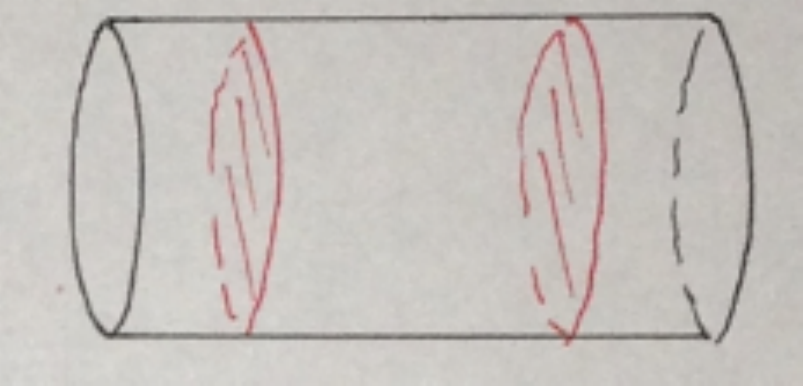}}
&$\longrightarrow$&
\fbox{\includegraphics[width=75pt]{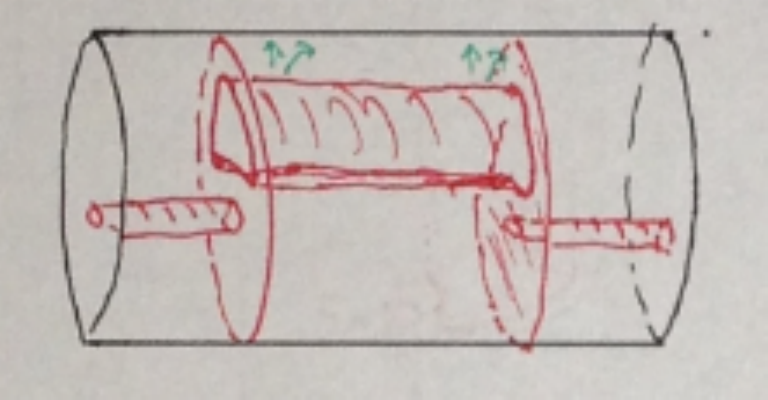}~~\includegraphics[width=75pt]{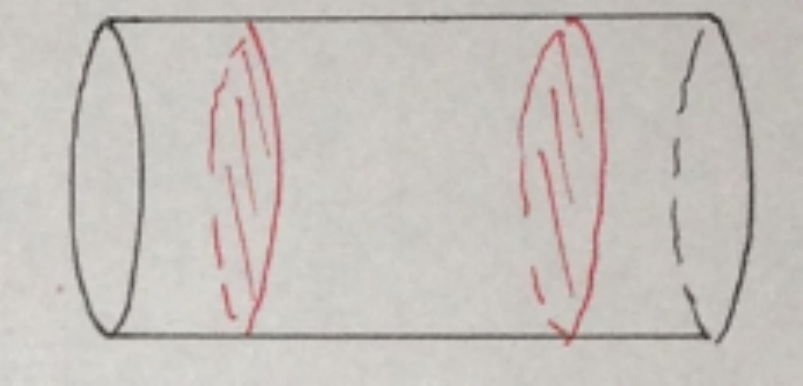}}
\\\vspace{0.4cm}\\$\longrightarrow$&
\fbox{\includegraphics[width=75pt]{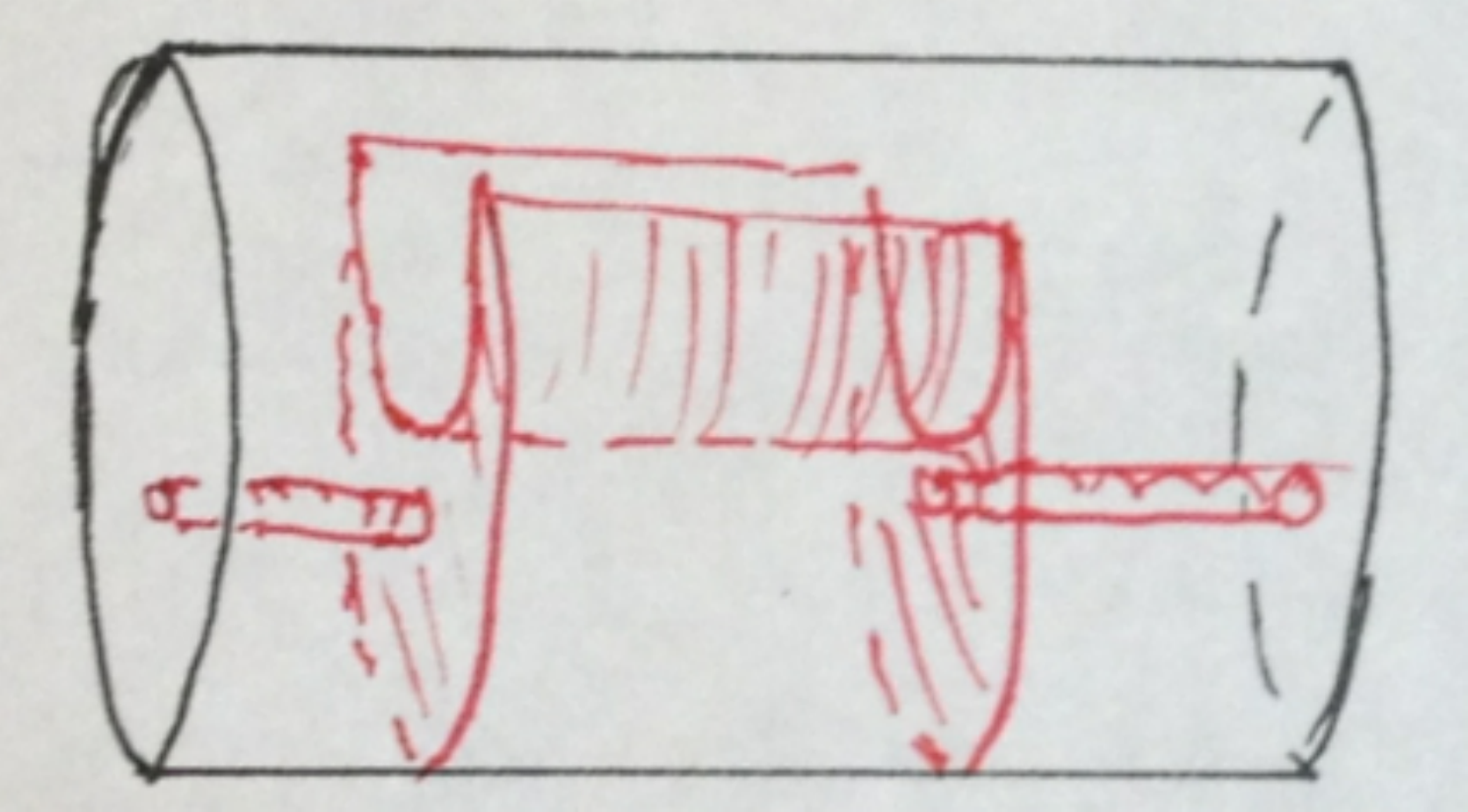}~~\includegraphics[width=75pt]{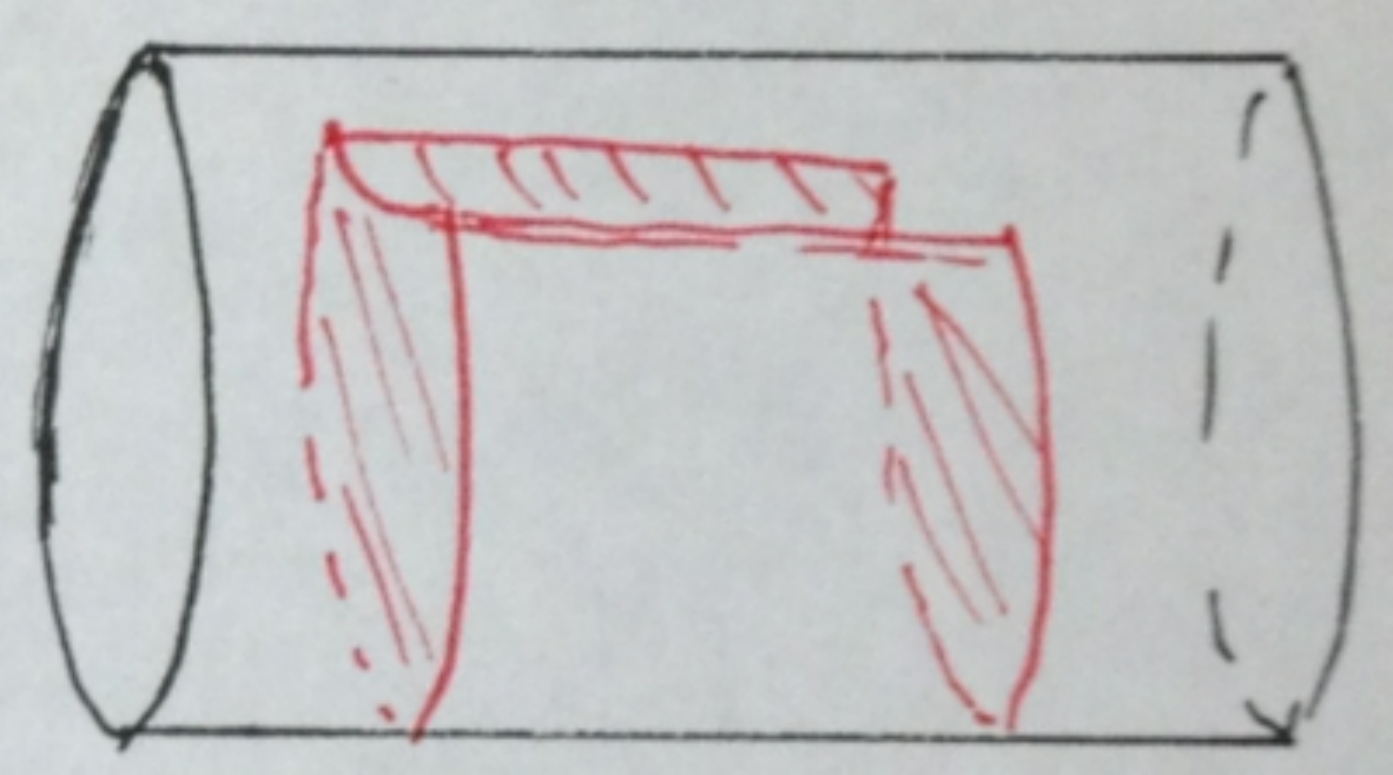}}
&$\longrightarrow$&
\fbox{\includegraphics[width=75pt]{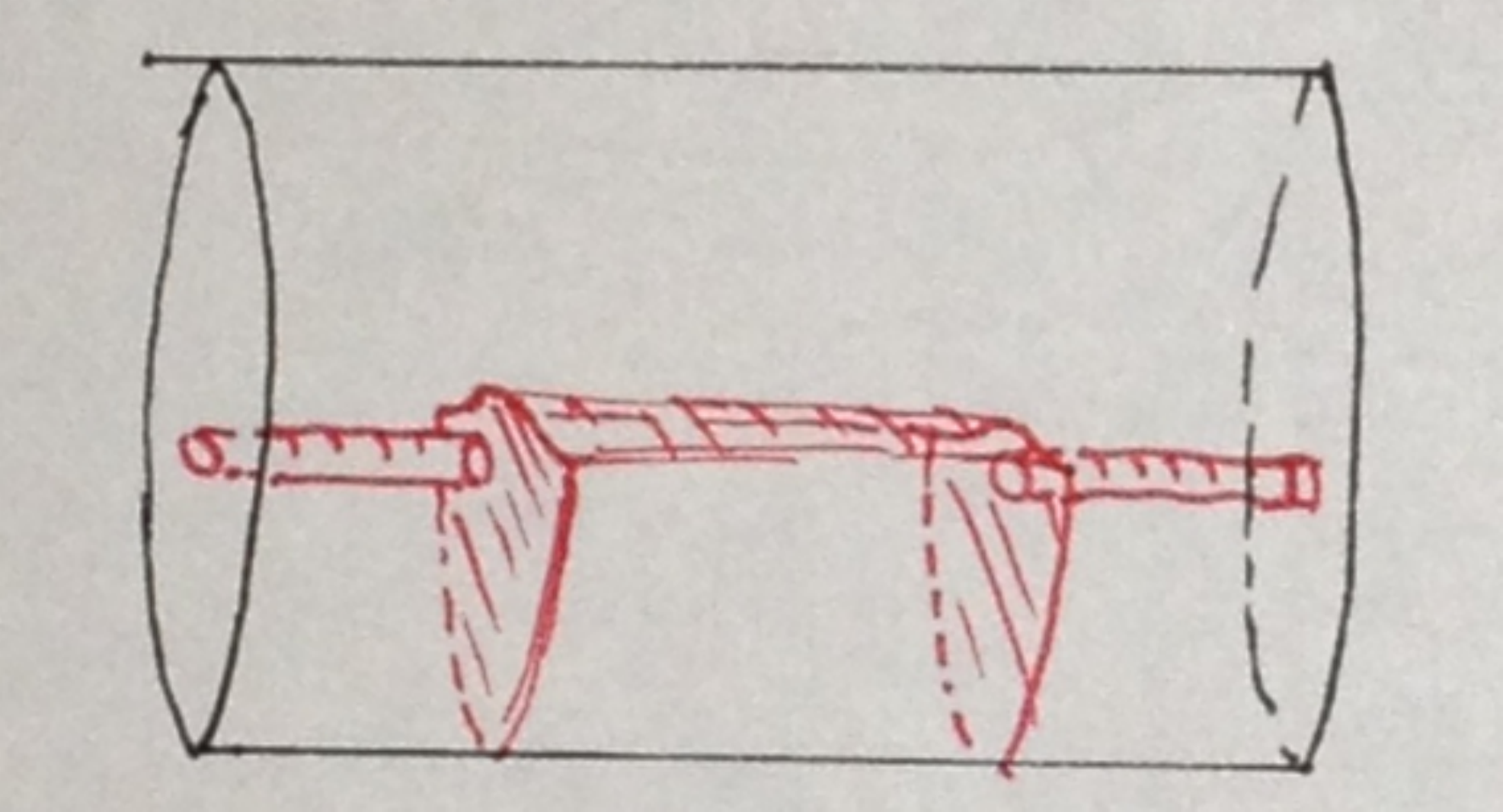}~~\includegraphics[width=75pt]{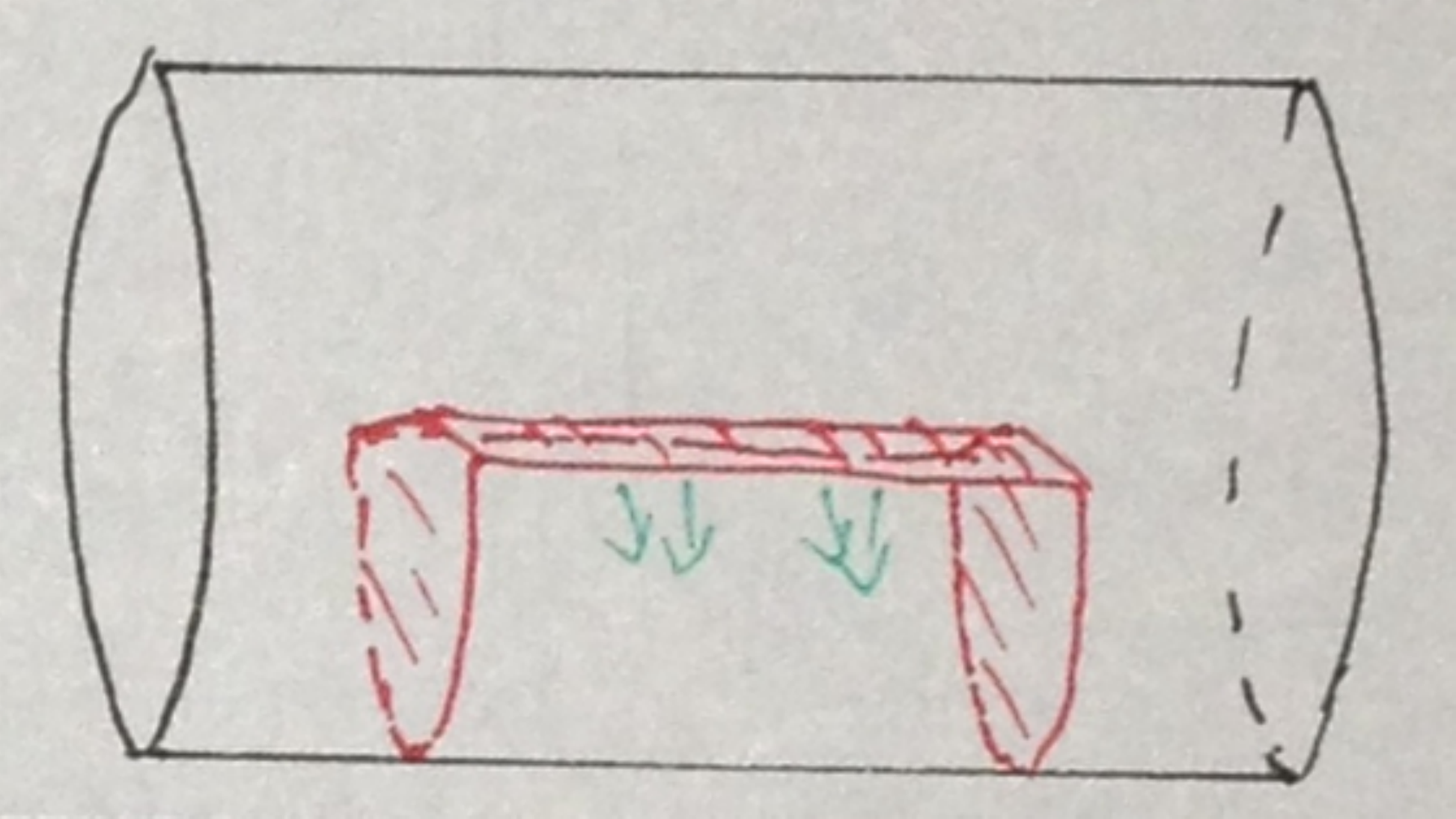}}
\\\vspace{0.4cm}\\$\longrightarrow$&
\fbox{\includegraphics[width=75pt]{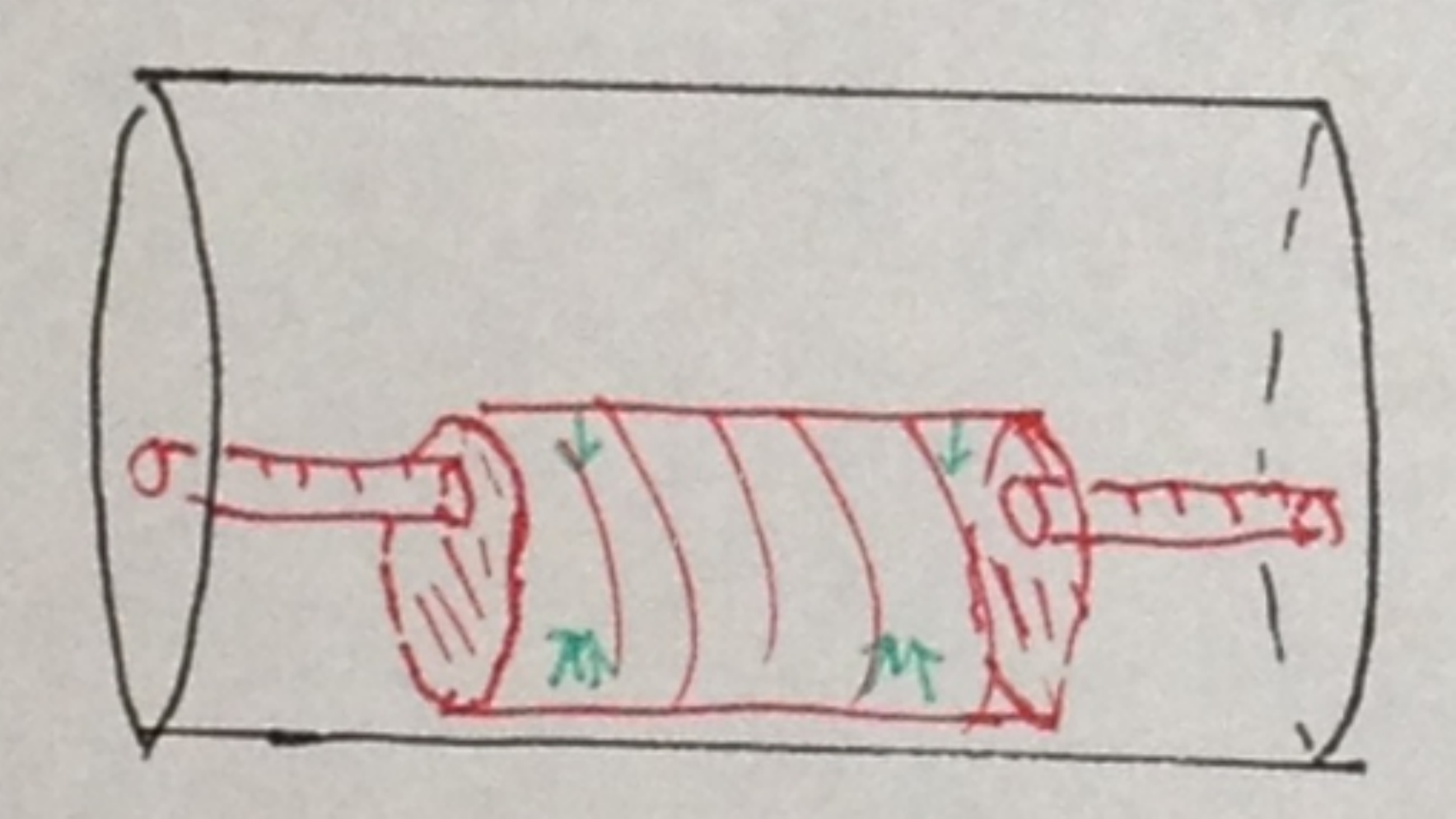}~~\includegraphics[width=75pt]{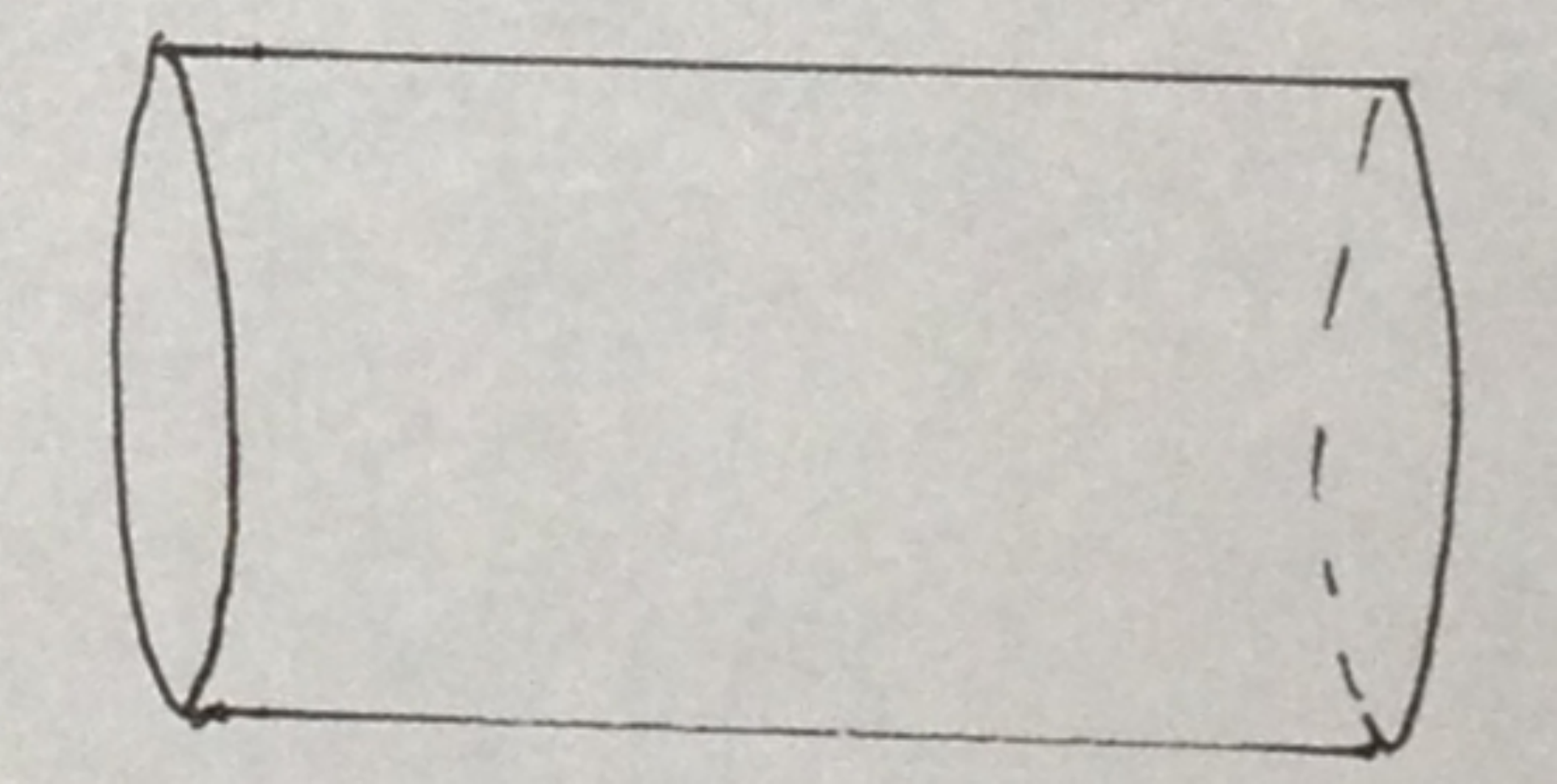}}
&$\longrightarrow$&
\fbox{\includegraphics[width=75pt]{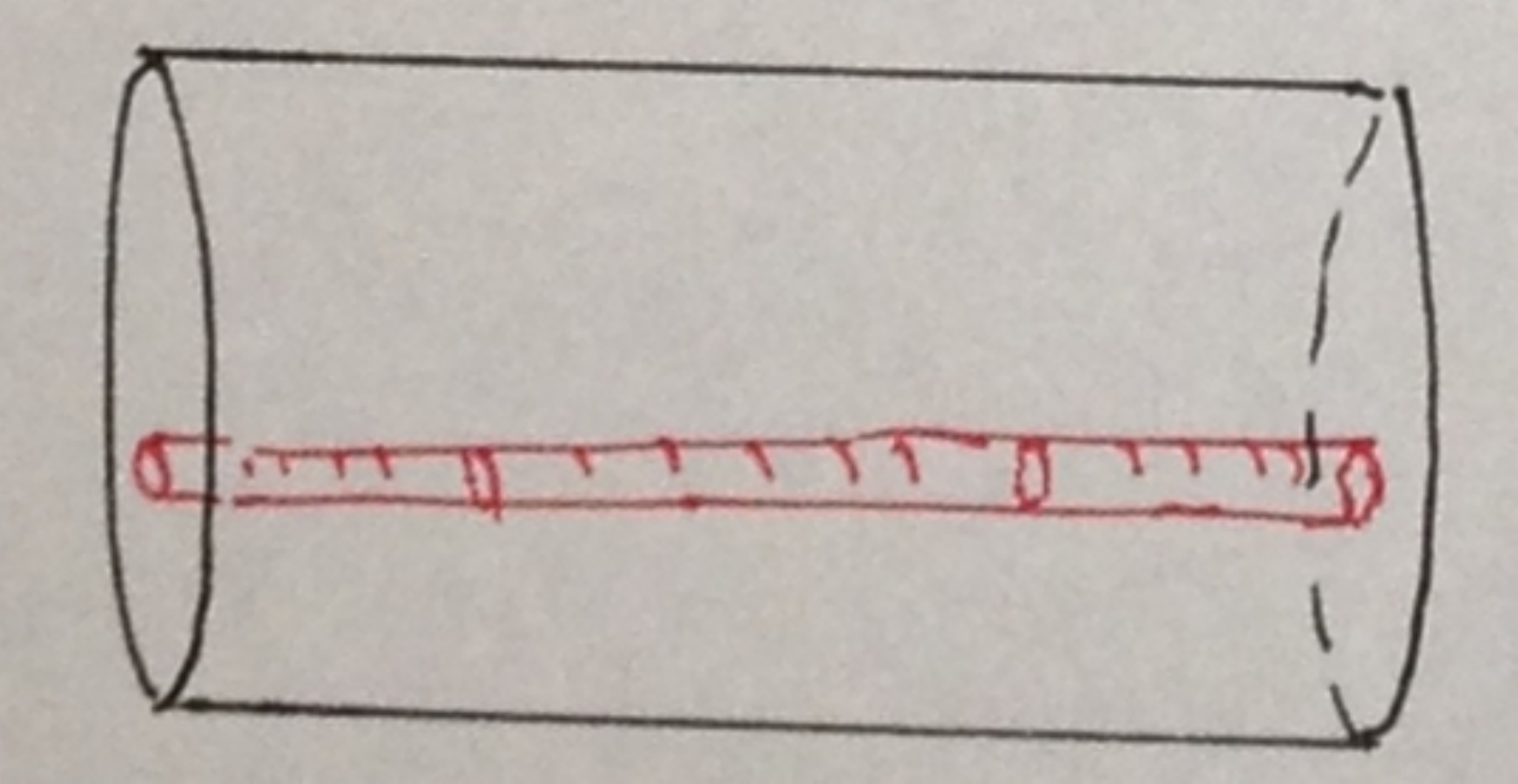}~~\includegraphics[width=75pt]{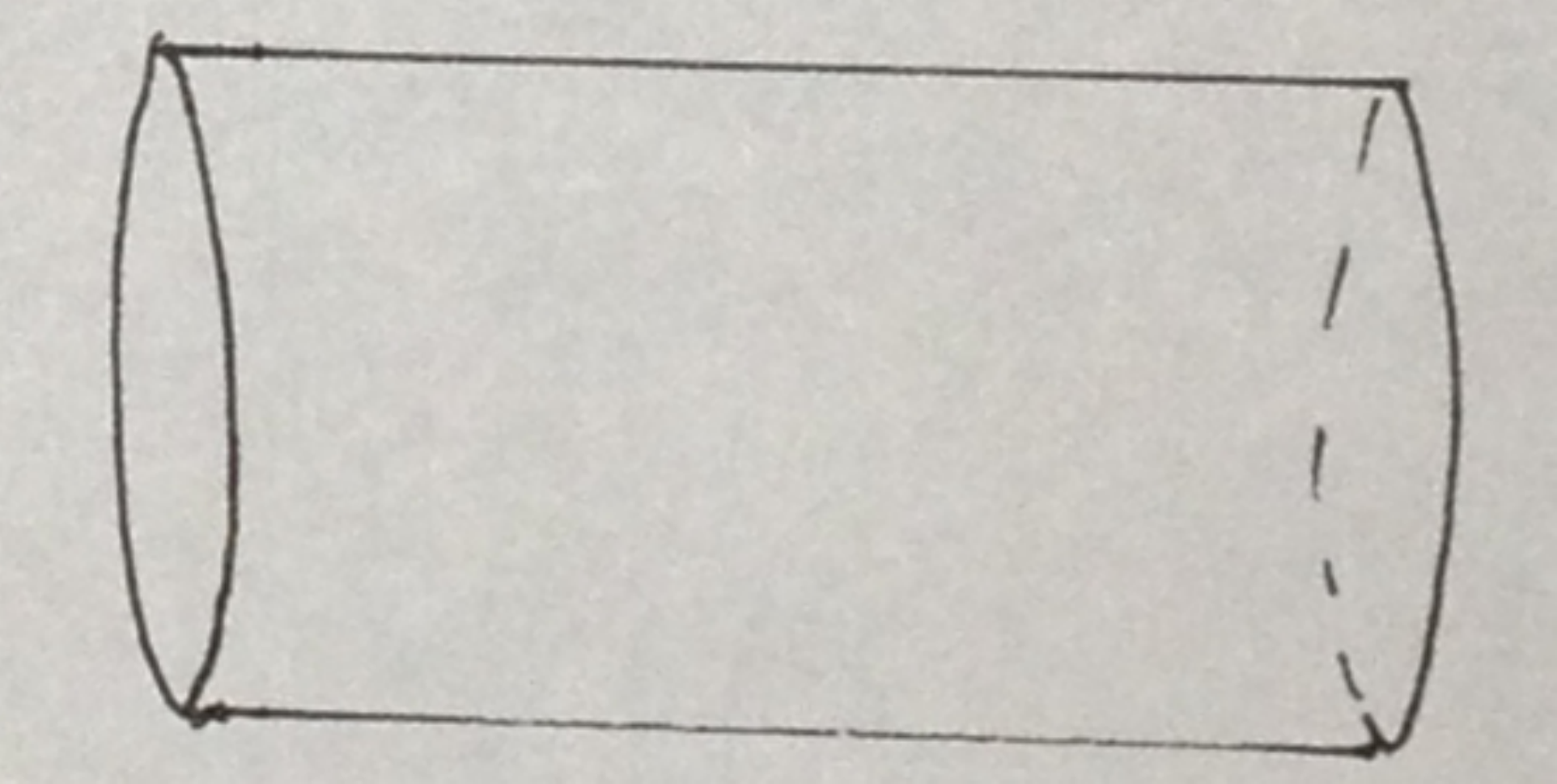}}
\\
\end{tabular}
\caption{The Bud Cancellation move on spheres.  Each box depicts two identical handlebodies, which are glued together via the identity map on their boundary to form the doubled handlebody.  Via this identification, the red embedded surfaces depicted glue up in each case to become an annulus.  The Bud Cancellation move exchanges the first box for the last; the pictures show how to interpolate with an isotopy.}
\end{center}
\end{figure}\todo{redraw}

Notice that, like Bud Exchange, Bud Cancellation move can be formulated for spheres in tree form:  two spheres in sphere tree form that differ by the inclusion or exclusion of two buds along a single edge of the universal cover $T$ of $\Gamma$ are homotopic, and are isotopic if no other buds lie between the given buds in $M_n$.

\begin{proof}
Two buds on the same edge connected by a twig are associated to the connected sum of two isotopic spheres via an annulus.  The connected sum of two embedded isotopic spheres is null-isotopic.  Let $S_{cancelled}$ denote the null-isotopic sphere associated to the sphere tree consisting of $b_1$ and $b_2$ connected by the twig between them.  Another twig connected to bud $b_i$ is associated to a surface with boundary connected to $S_{cancelled}$ by deleting a disk $D_i$ from $S_{cancelled}$ and identifying boundary components.  If only one of $b_1$ or $b_2$ is connected to another twig $t$ then the null-isotopy of $S_{cancelled}$ provides an isotopy of the resulting disk $S_{cancelled} - D_i$ to the deleted disk $D_i$.  The deleted disk $D_i$ is a cap on the surface with boundary associated to the twig $t$.  If both $b_1$ and $b_2$ are connected to twigs then $S_{cancelled} - D_1 - D_2$ is an annulus connecting these twigs together, the result of which is isotopic to a twig.  If $b_1$ and $b_2$ are innermost this local isotopy is in fact an isotopy of the sphere $S$ associated to $T_S$.  If $b_1$ and $b_2$ are not innermost then this local isotopy might introduce intersections with other buds of $S$ and so is only a homotopy.
\end{proof}

\subsection{Consolidated Sphere Trees}\label{sec:consolidated}

The isotopy conditions in the statements of the two moves on sphere trees suggest the an algorithm for simplifying a sphere tree, by applying the two simplification moves as much as possible.  To formally state this algorithm we need one more lemma.

\begin{lemma}\label{lem:existenceofcancellingbuds}
If $S$ is embedded and there exists an edge with more than one bud in $T_S$ then there exists a pair of buds in $T_S$ that are innermost with respect to each other.
\end{lemma}

\begin{proof}
By definition, if two buds $b_1$ and $b_2$ are on the same edge of a sphere tree and they are not innermost then there must exist some third bud $b_1'$ whose orbit contains a point between $b_1$ and $b_2$.  Consider the submanifold $E$ of $M_n$ corresponding to the nonsingular leaves of the foliation of $M_n$ lying over the edge containing $b_1$ and $b_2$, as in Section \ref{sec:spheretreesfromspheres}.  This submanifold is a product manifold $S^2 \times I$ where $I$ is an interval.  The spheres associated to $b_1$, $b_2$, and $b_1'$ are all fibers of this fibration.  We form $S$ in part by taking the connected sum of the spheres associated to $b_1$ and $b_2$, introducing an annular twig in $E$ connecting them.  Since $S$ is embedded this twig cannot intersect the sphere associated to $b_1'$, so $b_1'$ must have a twig adjacent to it.  Similarly, this new twig cannot intersect either $b_1$ or $b_2$, and so must terminate in a sphere associated to another bud $b_2'$.  The bud $b_2'$ is on the same edge of $T_S$ with $b_1'$.  Since $T_S$ has finitely many buds, iterating this argument must eventually terminate in two buds which are on the same edge and are innermost.
\end{proof}

We can now state the simplification algorithm.\\\\
\textbf{Sphere Tree Simplification Algorithm.} Assume $S$ is an embedded sphere in sphere tree form with sphere tree $T_S$ in a tree $T$.  Repeatedly apply the following two operations until neither applies.
\begin{enumerate}
\item If two buds of $T_S$ lie on the same edge and are innermost with respect to each other, apply the Bud Cancellation Move.
\item If a vertex of $T$ is adjacent to ends of $T_S$ in all but possibly one direction and these buds are innermost with respect to $v$, apply the Bud Exchange Move.\\
\end{enumerate}

A similar algorithm works for non-embedded $S$ (or if the resulting sphere need not be embedded) by ignoring the `innermost' requirements.

By Lemma \ref{lem:existenceofcancellingbuds}, this algorithm terminates in a sphere tree where no edge of $T_S$ has more than one bud, and where no vertex of $T_S$ is adjacent to ends of $T_S$ in all but possibly one direction.  We call a sphere tree \emph{consolidated} if the steps of the Sphere Tree Simplification Algorithm do not apply to it.  From now on we will assume all sphere trees are consolidated unless stated otherwise.  Furthermore, when applying the Bud Exchange Move to a sphere tree (e.g. to evolve sphere trees in the next section), if the Bud Cancellation Move applies to the result then we automatically apply it so that the result is still consolidated.  Checking the definition, note that a sphere associated to a consolidated sphere tree is always in Hatcher normal form.  Also note that many consolidated sphere trees can represent that same isotopy class of sphere, because the locations of the buds can vary by application of the Bud Exchange and Bud Cancellation moves.


\section{Quasigeodesics:  Folding Paths and Fold Paths}\label{sec:folding}

There are two important definitions related to approximating geodesics for proposed curve complex analogues:  that of the projection of a folding path, as used by Bestvina and Feighn to approximate geodesics in the factor complex \cite{bestvina_f:hyperbolicity_of_ff11}, and that of a fold path, as defined by Handel and Mosher in their proof of hyperbolicity of the splitting complex \cite{handel_m:hyperbolicity_of_fs12} and used there as approximations to geodesics.  For our purposes, we find the (continuous) notion of a folding path to be more relevant than the (discrete) notion of a fold path.  It comes as no surprise that these two notions are closely related. In this section, we introduce both folding paths and fold paths, and we show that a certain set of projections of folding paths also forms a nice family of quasigeodesics in the sphere complex.

\subsection{Definitions and Motivation}

Let $X$ be a simplicial complex.  A \emph{path of simplices} in $X$ is a sequence of simplices $\sigma_0, \dots, \sigma_k$ in $X$ such that for all $i$ either $\sigma_i$ is a face of $\sigma_{i+1}$ or vice versa.  We think of a path of simplices as a piecewise constant map from an interval into the simplices of $X$ that is continuous with respect to the poset topology.  This latter interpretation is indeed the case in the situations we care about:  projections of Teichm\"uller geodesics to the curve complex; projections of folding paths in outer space to the sphere and factor complexes; and fold paths, which can be interpreted as projections of certain paths in outer space to the sphere complex.

A family of paths of simplices is \emph{almost transitive path family} if there exists a constant $D$ such that, for any two vertices $u$ and $v$ of $X$, there exists a path of simplices $\sigma_0, \dots, \sigma_k$ in the family with $d(u,\sigma_0)\leq D$ and $d(v,\sigma_k) \leq D$.  A path of simplices $\sigma_0, \dots, \sigma_k$ is an \emph{unparametrized quasigeodesic} between $x$ and $y$ in $X$ if there exists constants $K > 0$ and $C \geq 0$ and a nondecreasing function $\rho: \mathbb{Z} \to \mathbb{Z}$ such that for all $0 \leq i, j \leq k$,
    $$\frac{1}{K}|\rho(i)-\rho(j)| - C \leq d_X(\sigma_{\rho(i)},\sigma_{\rho(j)}) \leq K|\rho(i)-\rho(j)| + C,$$
and moreover
    $$d_X(\sigma_{\rho(i)},\sigma_{\rho(i+1)}) \leq C.$$
A family of paths of simplices is a family of \emph{uniform unparametrized quasigeodesics} if the constants $K$ and $C$ can be chosen uniformly over the whole family.

As stated in the introduction, the motivation for the current approach to understanding the geometry of $\Out$ comes from the study of the mapping class group.  In that setting, Masur and Minsky \cite{masur_m:curve_complexI} proved:

\begin{theorem}[\cite{masur_m:curve_complexI}, Theorem 2.6.]
The set of projections of Teichm\"uller geodesics forms an almost transitive path family of uniform unparametrized quasigeodesics in the curve complex.
\end{theorem}

Bestvina and Feighn \cite{bestvina_f:hyperbolicity_of_ff11} and Handel and Mosher \cite{handel_m:hyperbolicity_of_fs12} have managed to prove analogues of this theorem in the case of the factor complex and the sphere (equivalently, splitting) complex, which we describe in this section.

But first, we need some standard terminology.

Let $\Gamma$ be a simplicial graph (possibly infinite, possibly locally infinite, so in particular possibly an $\mathbb{R}$-tree).  A \emph{natural vertex} of $\Gamma$ is a vertex of degree at least 3, and a \emph{natural edge} is the closure of a component of the complement of the set of natural vertices in $\Gamma$.  A \emph{direction} at a point $x \in\Gamma$ is the germ of a non-degenerate embedded segment in $\Gamma$ beginning at $x$.  A \emph{turn} at $x$ is an unordered pair of distinct directions at $x$.  A subset of $D_x$ is called a \emph{gate} at $x$.  A \emph{train track structure} on $\Gamma$ is a partition of $D_x$ for each vertex $x$ into at least two gates.  A turn is \emph{illegal} with respect to a given train track structure if both directions are contained in the same gate, and \emph{legal} otherwise.  A path in $\Gamma$ is called \emph{legal} if the train track structure on the path induced by inclusion into $\Gamma$ has no illegal turns.

Now let $\Gamma'$ be another simplicial graph.  A \emph{morphism} $\phi \co \Gamma \to \Gamma'$ is a map such that every natural edge, which is isometric to an interval, can be subdivided into subintervals on which $\phi$ is an isometric embedding.  Each morphism $\phi$ induces a partition on the set of all directions at each vertex $x$ of $\Gamma$, where two directions $d$ and $d'$ at $x$ are in the same partition set if $D\phi_x(d) = D\phi_x(d')$.  If this partition defines a train-track structure on $\Gamma$, the morphism $\phi$ is called a \emph{train track map}.

\subsection{Outer Space and its Connection to the Sphere Graph}
\label{ssec:outer_space}

For us, all relevant graphs that will be involved in train track maps come from \emph{outer space}.  Outer space is a topological space first defined by Culler and Vogtmann in their seminal paper~\cite{culler_v:outer_space}, and should be considered an analogue to Teichm\"uller space for the mapping class group.  For more information about outer space we  refer the reader to the excellent survey article~\cite{vogtmann:aut_fn_and_outer_space}. We take the definitions below mostly from that paper. Note that for our purposes we will use an unprojectivized version of an outer space.

Let $R_n$ be the graph with one vertex and $n$ edges (we call such a graph a \emph{rose}). We will identify the free group $F_n$ with the fundamental group $\pi_1(R_n)$ of $R_n$ in such a way that the generators of $F_n$ correspond to single oriented edges of $R_n$.

\begin{definition}
The \emph{(unprojectivized Culler-Vogtmann) outer space} $\cv$ is the space whose points are equivalence classes of pairs $(\tau,\Gamma)$ where:
\begin{itemize}
\item $\Gamma$ is a graph with fundamental group $F_n$;
\item each edge of $\Gamma$ is assigned a positive real length, making $\Gamma$ into a metric space via the path metric;
\item each vertex of $\Gamma$ has degree at least $3$;
\item $\tau\colon R_n\to\Gamma$ is a homotopy equivalence, called the \emph{marking}; and
\item two pairs $(\tau,\Gamma)$ and $(\tau',\Gamma')$ are equivalent if and only if there is an isometry $h\colon\Gamma\to\Gamma'$ such that $h\circ \tau$ is homotopic to $\tau'$.
\end{itemize}
A pair $(\tau, \Gamma)$ is called a \emph{marked metric graph}.
\end{definition}

\begin{definition}
The \emph{projectivized (Culler-Vogtmann) outer space} $\CV$ is the quotient of $\cv$ by the equivalence relation induced by scaling the graphs.
\end{definition}

Equivalently, one may think of points in $\CV$ as marked metric graphs in which the sum of the lengths of all edges is equal to one.

For each $x=(\tau,\Gamma)\in\cv$ the homotopy equivalence $\tau$ induces the isomorphism $\tilde\tau$ from $\pi_1(R_n)$, identified with $F_n$, to $\pi_1(\Gamma)$. Conversely, given an isomorphism $\eta\colon F\to\pi_1(\Gamma)$, it defines a homotopy equivalence $\tau$ from $R_n$ to $\Gamma$, and thus defines a point $x=(\tau,\Gamma)\in\cv$ such that $\tilde\tau=\eta$ (see, for example,~\cite{kapovich_n:subset_currents12}). Note that different automorphisms from $F_n$ to $\pi_1(\Gamma)$ may induce homotopic markings, and thus represent the same point of $\cv$. For example, if $\Gamma=R_n$ then inner automorphisms of $F_n$ induce markings homotopic to the trivial one.

The group $\Out$ acts on $\cv$ on right by changing the marking: given $\phi\in\Out$, let $f\colon R_n\to R_n$ be a representative for $\phi$; then $(\tau,\Gamma)\phi=(\tau\circ f,\Gamma)$. On the level of isomorphisms, this action simply corresponds to the right multiplication: $(\eta,\Gamma)\phi=(\eta\phi,\Gamma)$.

Alternatively and equivalently, we can think of points of the unprojectivized outer-space as free minimal actions of $F_n$ on simplicial $\R$-trees. Each element $(\tau,\Gamma)\in\cv$ induces such an action of $F_n$ on the universal cover $\tilde\Gamma$ of $\Gamma$ via the identification $\hat\tau$ between $F_n$ and $\pi_1(\Gamma)$ and the action of the fundamental group $\pi_1(\Gamma)$ on $\Gamma$ by deck transformations. Conversely, each free minimal action of $\F$ on the simplicial $\R$-tree induces a marking of the quotient space of this action, which is a graph with fundamental group isomorphic to $\F$. In this paper it will be convenient to use both these viewpoints. In particular, we will describe points in $\cv$ sometimes as graphs with markings, and sometimes as $\R$-trees with actions of $F_n$.

Projectivized outer space $\CV$ can be endowed with a non-symmetric Lipschitz metric. Let $(\tau,\Gamma)$ and $(\tau',\Gamma')$ be two points in $\CV$. We will call a morphism $\phi\colon\Gamma\to\Gamma'$ a \emph{difference of markings} if $\tau'$ is homotopic to $\phi\tau$. The \emph{Lipschitz distance} between $(\tau,\Gamma)$ and $(\tau',\Gamma')$ is the $\log$ of the minimal Lipschitz constant over all differences of markings from $\Gamma$ to $\Gamma'$. For more information on this notion, see~\cite{francaviglia_m:mertic_properties_of_outer_space11}.

There is a deep connection between $\CV$ and $\SC(M_n)$ described by Hatcher in the Appendix of~\cite{hatcher:homological_stability95}. Namely, if one denotes by $\SC(M_n)_\infty$ the subcomplex of $\SC(M_n)$ consisting of non-minimal sphere systems, then $\SC(M_n)-\SC(M_n)_\infty$ is homeomorphic to $\CV$. The rough idea is the following: given a minimal sphere system $[\hat S]$ in $\SC(M_n)-\SC(M_n)_\infty$ and a fixed rose $R_n$ embedded into $M_n$, one constructs a marked metric graph $(\tau_{[\hat S]},\Gamma_{\hat S})\in \CV$ as a dual graph to $[\hat S]$ in $M_n$, where the homotopy equivalence $\tau_{[\hat S]}$ from $R_n$ to $\Gamma_{[\hat S]}$ is defined as a composition of embedding of $R_n\to M_n$ and the collapse map $M_n\to\Gamma_{\hat S}$. Conversely, given a point $(\tau,\Gamma)$ in $\CV$ one constructs a 3-manifold $M_{\Gamma}$ diffeomorphic to $M_n$ by thickening $\Gamma$. The marking $\tau$ induces a diffeomorphism from $M_n$ to $M_{\Gamma}$ and the simple sphere system $[\hat S]$ in $M_n$ is defined as the preimage of the simple sphere system in $M_{\Gamma}$ corresponding to midpoints of edges in $\Gamma$. For more details we refer the reader to~\cite{hatcher:homological_stability95}.

This homeomorphism between $\SC(M_n)-\SC(M_n)_\infty$ and $\CV$ induces a continuous onto map from $\cv$ to $\SC(M_n)$.

\subsection{Folding Paths}

We now recall the definitions related to folding paths along with some standard facts about them.  These are proven for instance in~\cite{francaviglia_m:mertic_properties_of_outer_space11}.

\begin{definition}
Let $T$ and $T'$ represent two $\R$-trees in unprojectivized outer space $\cv$ and let $\phi \co T \to T'$ be a morphism inducing a train track structure on $T$.  For every $t \geq 0$ let $\sim_t$ denote an equivalence relation on points of $T$, where $x \sim_t y$ iff $\phi(x) = \phi(y)$ and $d_T(x,y) \leq t$.  Let $T_t$ denote the quotient of $T$ by the equivalence relation $\sim_t$.  Then $T_t$ is a tree, $\phi$ factors through $T_t$, and $T_t$ carries an induced free minimal action of $\F$ and so is a point in outer space.  In this case we say that the tree $T_t$ is \emph{obtained from $T$ by folding (all) illegal turns at unit speed for time $t$ with respect to $\phi$}.
\end{definition}

\begin{proposition}\label{prop:foldingpath}
Let $T, T' \in \cv$.  For any morphism $\phi \co T \to T'$ such that $\phi$ induces a train track structure on $T$, there exists a unique continuous path $(T_t)$, $t \in [\alpha,\omega]$ with morphisms $\phi_{st} \co T_s \to T_t$ for $s \leq t$ such that:
\begin{enumerate}
\item $T_\alpha = T$ and $T_\omega = T'$,
\item $\phi_{\alpha\omega} = \phi$,
\item $\phi_{tt} = Id$ for all $t$,
\item $\phi_{st} = \phi_{su}\phi_{ut}$ for $s \leq u \leq t$,
\item each $\phi_{st}$ isometrically embeds edges and induces a train track structure on $T_s$,
\item for $s < t, t'$ the illegal turns of $T_s$ with respect to $\phi_{st}$ and with respect to $\phi_{st'}$ coincide, so $T_s$ has a well-defined train track structure independent of $t$, and
\item for every $s$ there exists $\epsilon > 0$ such that $T_{s+\epsilon}$ is obtained from $T_s$ by folding illegal turns at unit speed.
\end{enumerate}
\end{proposition}

The path $(T_t)$ in $\cv$ is called the \emph{folding path} associated to $\phi$.  Note that we do not rescale the quotient graphs $F_n\backslash T_t$ to all have metric volume $1$, instead allowing their metric volume to monotonically and continuously decrease along a folding path. With a slight abuse of notation we will call the projection of a folding path in $\cv$ to $\CV$ also by folding path.

In order to relate folding paths to geodesics in $\CV$ with respect to the Lipschitz metric we need to introduce a notion of an \emph{optimal map}.

\begin{definition}
Let $\Gamma, \Gamma' \in \CV$, with markings $\tau$ and $\tau'$, respectively.  A map $\phi \co \Gamma \to \Gamma'$ is \emph{optimal} if:
\begin{enumerate}
\item $\phi$ is a difference of markings;
\item for each edge of $\Gamma$ the restriction of $\phi$ is either constant or an immersion with constant speed (called the \emph{slope} of $\phi$ on the given edge);
\item the set of edges on which $\phi$ has maximal slope has no vertices of degree 1 (this set of edges is the \emph{tension subgraph} of $\Gamma$); and
\item $\phi$ induces a train track structure on the tension subgraph.
\end{enumerate}
\end{definition}

\begin{proposition}\label{prop:optimalmap}
Let $\Gamma, \Gamma' \in \CV$.  There exists an optimal map from $\Gamma$ to $\Gamma'$.
\end{proposition}

Note that optimal maps are not unique (unlike Teichm\"uller maps in Teichm\"uller space).

\begin{proposition}[\cite{francaviglia_m:mertic_properties_of_outer_space11}]
For each $\Gamma,\Gamma'\in\CV$ there is $\Gamma''\in\CV$ in the closure of the same simplex as $\Gamma$ and an optimal map $\phi\colon\Gamma''\to\Gamma'$, such that the following path is a geodesic from $\Gamma$ to $\Gamma'$:  first follow the line from $\Gamma$ to $\Gamma''$ in their simplex, and then follow the folding path from $\Gamma''$ to $\Gamma'$ associated to the lift $\tilde\phi$ of $\phi$ to universal covers.

Moreover, if there is an optimal map $\phi$ from $\Gamma$ to $\Gamma'$ with tension subgraph equal to $\Gamma$, then $\Gamma'' = \Gamma$ and the folding path from $\Gamma''$ to $\Gamma'$ is a geodesic.
\end{proposition}

\subsection{Fold Paths}

We now turn to the notion of a fold path as defined by Handel and Mosher \cite{handel_m:hyperbolicity_of_fs12}.  We begin by recalling the relevant definitions.  A more extended exposition of the definitions and facts concerning fold paths can be found in \cite{handel_m:hyperbolicity_of_fs12}.  Given two $\R$-trees $S, T \in \cv$ such that $S$ has no vertices of degree $2$, a map $f\co S \to T$ is \emph{foldable} if $f$ is injective on each edge of $S$ and $f$ has at least 3 gates at each vertex of $S$.  A \emph{maximal fold factor} of a map $f \co S \to T$ is a map $h \co S \to U$ such that $f$ factors through $h$ and $h$ is the identity on $S$ except for the following property:  $h$ equivariantly folds together exactly two $F_n$-orbits of oriented initial segments of edges such that the initial segments have maximal length. If a map $f$ is foldable then after performing the maximal fold factor the induced map from the quotient space to $T$ is also foldable. Therefore, we can consider sequences of maximal fold factors, preserving the fact that we have a foldable map at each step. A \emph{fold sequence} is a sequence of trees $S_0, \dots, S_K \in \cv$ together with foldable maps $f^i_j \co S_i \to S_j$ such that $f^i_{i+1}$ is a maximal fold factor of $f^i_K$ for all $i = 1, \dots, K-1$.  A \emph{fold path} is a sequence of simplices in $\SC(M_n)$ which is the projection of a fold sequence (defined in the end of Subsection~\ref{ssec:outer_space}).

Note fold paths are defined in \cite{handel_m:hyperbolicity_of_fs12} for arbitrary minimal simplicial actions of $F_n$ on trees with trivial edge stabilizers, not just points from the interior of outer space, though we do not need that greater generality here (similar to the point of view exploited in~\cite{kapovich_r:hyperbolicity_FF12}).

Now we recall some facts about fold paths.  These facts follow from (the proofs of) results in Section 2 of \cite{handel_m:hyperbolicity_of_fs12}.  For two consecutive simplices in a fold path, the simplices are distinct but share a common face (which can be one of the two simplices).  For any two trees $S, T \in \cv$ with $S$ having no vertices of degree 2, there exists trees $S', S'' \in \cv$ such that both have no vertices of degree 2, $S$ and $S''$ both differ from $S'$ by equivariantly collapsing some subset of $S'$, and there exists a foldable map $S'' \to T$.  Note that when $S$ is from outer space (and not its compactification), we may take $S = S'$, and so the projection of $S''$ to $\SC(M_n)$ is a face of the projection of $S$ to $\SC(M_n)$.  Moreover, when $S'$ has $F_n$-quotient a rose and $T$ is from outer space (and not its compactification), $S'' = S'$.  For any foldable map $f\co S \to T$ there exists a fold sequence $S=S_0, \dots, S_K=T \in \cv$ such that $f = f^0_K\co S \to T$.

For our purposes, the most important property of fold paths, shown by Handel and Mosher, is that they are quasigeodesics in the (hyperbolic) sphere complex.

\begin{proposition}\cite{handel_m:hyperbolicity_of_fs12}
The set of all fold paths forms an almost transitive path family of uniform unparametrized quasigeodesics.
\end{proposition}

\subsection{Terse Paths}\label{sec:terse}

Fold paths come from discrete sequences of points in outer space.  However, our proof of the Bounded Geodesic Image theorem requires continuous paths in outer space that project to quasigeodesics.  We find these by looking at a particular kind of hybrid between a fold path and a folding path, which we call a \emph{terse path}.  Conceptually, a terse map is an optimal map which has full tension subgraph and is also a foldable map, so that the projection of the associated folding path to the sphere complex parallels a fold path.  We choose the name terse because a foldable map can be constructed by eliminating all unnecessary and potentially time-wasting edges in the relevant graphs, an optimal map wastes no time on backtracking, and having full tension subgraph means all edges are being folded as fast as possible.

\begin{definition}
Let $T,T'$ be $\R$-trees in $\cv$ and let $\Gamma,\Gamma'\in\CV$ be the projections of $T,T'$. A \emph{terse map} from $T$ to $T'$ is a foldable map $\tilde\phi\co T\to T'$ induced by the optimal map $\phi$ from $\Gamma$ to $\Gamma'$, whose tension graph is all of $\Gamma$. A \emph{terse path} from $T$ to $T'$ is a folding path associated to a terse map from $T$ to $T'$.
\end{definition}

\begin{lemma}\label{lem:terse_paths_project_to_geodesics}
The set of projections of terse paths forms an almost transitive path family of uniform unparametrized quasigeodesics in the sphere complex.
\end{lemma}

Recall that a sphere system is \emph{reduced} if its complement is simply connected.

\begin{proof}
Every vertex of the sphere complex is a face of infinitely many simplices which are represented by reduced sphere systems, so to prove these paths form an almost transitive path family it suffices to find a terse path whose projection connects any two such simplices.

For any sphere system $A$, let $[A]$ denote the simplex in $\SC$ represented by $A$.

For any two sphere systems $A$ and $B$, Handel and Mosher prove \cite[Lemma 2.3]{handel_m:hyperbolicity_of_fs12} that there exists a sphere system $A''$ such that:
\begin{itemize}
\item $[A]$ and $[A'']$ are faces of a common simplex in $\SC$, and
\item there exists a foldable map $T_{A''} \mapsto T_B$, where $T_{A''}$ and $T_B$ are points in unprojectivized outer space $\cv$ whose projections to the sphere complex are $[A'']$ and $[B]$, respectively.
\end{itemize}
Their proof shows that if $A$ is simple (that is, $T_A$ is locally finite) then $A''$ can be taken to represent a face of $[A]$, and that if $A$ and $B$ are reduced sphere systems (that is, $T_A$ and $T_B$ are locally finite with $F_n$-quotient a rose having exactly $n$ edges) then $A''$ can be taken to be $A$.  Thus, given any two reduced sphere systems $A$ and $B$ there exists a foldable map $f: T_A \to T_B$ for trees $T_A$ and $T_B$ projecting to $[A]$ and $[B]$, respectively.  Moreover, this map $f$ may be taken to be optimal.

Let $A$ and $B$ be reduced sphere systems, and let $f: T_A \to T_B$ be as described.  The property of being a foldable map only depends on the train track structure on $T_A$ induced by $f$, and does not depend on the metric of $T_A$.  Thus, given a foldable map, we may rescale tree $T_A$ and the map $f$ until the tension subgraph of $f$ is all of $T_A$, by defining a metric on $T_A$ so that $f$ restricts to an isomorphism on every preimage of every edge of $T_B$.  Note that $f$ is nonconstant on every edge of $T_A$ since $A$ and $B$ are reduced, so $T_A$ still projects to $[A]$ (though even if $f$ were constant on some edges, assigning some edges length zero would not decrease the number of gates at any vertex, preserving foldability of $f$ with $T_A$ projecting to a face of $[A]$).  Thus, between any two reduced sphere systems there exists a projection of a terse path, and so the set of projections of terse paths is an almost transitive path family.

It remains to show that terse paths project to unparametrized quasigeodesics with uniform constants in $\SC(M_n)$.  To show that a given terse path projects to an unparametrized quasigeodesic, we construct a fold path that is uniformly close to it.  Then, as Handel and Mosher have shown that fold paths are uniform unparametrized quasigeodesics \cite{handel_m:hyperbolicity_of_fs12}, projections of terse paths will be too.

We begin constructing this fold path by considering the behavior of illegal turns along a terse folding path $(T_t)_{t\in [0,N]}$.  For some $t \in [0,N]$, let $\Gamma_t := F_n \backslash T_t$ denote the quotient graph and let $\omega$ denote some gate at a vertex $v$ in $\Gamma_t$.  Under folding, there is some nontrivial amount of time such that the vertex $v$ and the directions in $\omega$ at $v$ all evolve continuously and in a well-defined manner.  This continuous evolution only stops when either $v$ evolves to collide with some other vertex or the gate $\omega$ splits into multiple gates (or both).  This time is called the \emph{critical time} for the gate $\omega$.

Along a folding path between points in outer space, the set of all critical times for all possible gates -- that is, the set of times when two vertices collide or when a gate splits -- is finite and hence discrete.  Let $t_1 < \dots < t_k$ denote the critical times listed in order, and set $t_0 := 0$.  For each $i = 0, \dots, k-1$, let $\Omega_i$ denote the set of all gates have critical time $t_{i+1}$. Evolving each gate $\omega$ in $\Omega_i$ backward in time, there exists some minimal time $t' < t_{i+1}$ and some gate $\omega'$ in $\Gamma_{t'}$ such that $\omega'$ evolves continuously to $\omega$.  By minimality, it must be that $t' = t_j$ for some $0 \leq j \leq i$. Thus, we can think of each gate $\omega\in\Omega_i$ as existing in each $\Gamma_t$ for $t_j\leq t\leq t_{i+1}$. Construct the graph $\hat\Gamma_i$ by maximally folding all gates from $\Omega_1,\Omega_2,\ldots,\Omega_{k-1}$ in $\Gamma_{t_0}$. Note, that by definition of $\Omega_{i+1}$ some gates in $\Omega_{i+1}$ might be created only after folding all gates in $\Omega_j$, $j\leq i$. By construction, we get that $\Omega_{i+1}$ is the set of gates in $\hat\Gamma_{i}$ and that $\hat\Gamma_{i+1}$ is obtained from $\hat\Gamma_i$ by maximally folding all gates in $\Omega_{i+1}$ viewed as gates in $\hat\Gamma_i$.  Moreover, one can obtain $\Gamma_{t_i}$ by (non-maximally) folding gates from $\Omega_{i+1},\Omega_{i+2},\ldots,\Omega_{k}$ that are defined in $\hat\Gamma_i$ up to a time $t_i$.

We now claim that the projections of the sequences $\{\Gamma_t\}$ and $\{\hat \Gamma_i\}$ are quasiisometric.  First, by the last remark in the previous paragraph, the graph $\Gamma_{t_i}$ can be obtained from $\hat\Gamma_i$ by folding some gates that never collide with any vertices. This means that $\Gamma_{t_i}$ projects in $\SC(M_n)$ to a simplex whose face is the projection of $\hat\Gamma_i$. In particular, the projections of $\hat\Gamma_i$ and $\Gamma_{t_i}$ are at most distance one apart. On the other hand, by the same reasoning, the distance between projections of $\Gamma_{t_i}$ and $\Gamma_{t}$ is also at most one for $t\in[t_i,t_{i+1})$. Thus the claim follows.

Finally, we claim that the sequence of graphs $\{\hat \Gamma_i\}$ can be interpolated to a fold sequence such that the projections of $\{\hat \Gamma_i\}$ and the fold sequence are quasiisometric.  But this follows immediately from construction, since the number of turns in each $\Omega_i$ is uniformly bounded above by a function of $n$, and from the definition of a fold sequence, since folding each $\Omega_i$ maximally can be realized by folding individual turns within $\Omega_i$ maximally, and by Lemma~2.5 in~\cite{handel_m:hyperbolicity_of_fs12} performing a maximal fold cannot move more than distance than 2. The lemma then follows.
\end{proof}

\subsection{Evolving Sphere Trees along Folding Paths}\label{sec:evolving}

In this paper, we prefer to use folding paths over fold paths because with the continuously varying folding paths it is easier to keep track of how sphere trees change along the path (see in particular the proof of the Bounded Geodesic Image Theorem, which requires the terse folding paths of Section \ref{sec:terse}).  In this section we detail that continuous variation, which we call \emph{evolving} a sphere tree along a folding path.  Note that the evolution of a sphere tree is not the same operation as taking the image of the tree as a set under any associated optimal map.

It is shown in Section~\ref{sec:spheretrees} how, given a simple sphere system $\mathcal A$ in $M_n$ and a sphere $S$ in $M_n$, one can construct a consolidated sphere tree $T_S$ of $S$ in the tree $T$, which is a universal cover of a dual graph $\Gamma$ to the sphere system $\mathcal A$, viewed as a point in $\cv$. Moreover, by Hatcher's correspondence discussed in the end of Subsection~\ref{ssec:outer_space} for any point $T\in\cv$ one can associate a simple sphere system $\mathcal A_T$ in $M_n$ and, hence, the consolidated sphere tree $T_S$ contained in $T$. Thus, given a folding path $(T_t),0\leq t\leq M_n$ in $\cv$, and a sphere $S$ in $M_n$ one can ask how the consolidated sphere trees $(T_t)_S$ evolve inside the trees $T_t$ along the folding path.

Moving along a portion of the folding path where vertices do not collide corresponds simply to rescaling the metric on $T_t$ and has no effect on the sphere trees $(T_t)_S$ that evolve continuously with $T_t$. The problem arises when two vertices $v$ and $w$ in $T_t$ collide at some time between times $t < t'$ when folding from time $t$ to time $t'$, and there is a bud of $(T_t)_S$ between $v$ and $w$, because by definition of sphere trees we do not allow buds to coincide with vertices of $T_t$. To avoid this problem we will simply use the Bud Exchange Move. First, we assume that time $t$ was chosen in such a way that no other vertices collide with either $v$ or $w$ between times $t$ and $t'$ (think of $t$ and $t'$ as chosen to be $\epsilon$ before and after to the collision time, respectively, for some very small $\epsilon$).  One can exchange the set of buds adjacent to one of the vertices that collide to its complement, which creates a new sphere tree $(T_t)_{S'}$ (possibly not consolidated) that corresponds to the sphere $S'$ homotopic to $S$ in $M_n$. After this exchange there will be no buds in $(T_t)_{S'}$ between $v$ and $w$, since $(T_t)_S$ was assumed to be consolidated.  Hence $(T_t)_{S'}$ can evolve continuously while the vertices $v$ and $w$ collide along the folding path to a sphere tree $(T_{t'})_{S'}$. Finally, after the collision we simply reduce $(T_{t'})_{S'}$ to consolidated form by applying Bud Cancellation and Bud Exchange moves as necessary to obtain $(T_{t'})_S$. An example of such evolution is shown in Figure~\ref{fig:sphere_evolution}.

\begin{figure}[h]
\begin{center}
\begin{tabular}{m{10pt}m{150pt}m{15pt}m{150pt}}
&
\fbox{\minibox{\includegraphics[width=140pt]{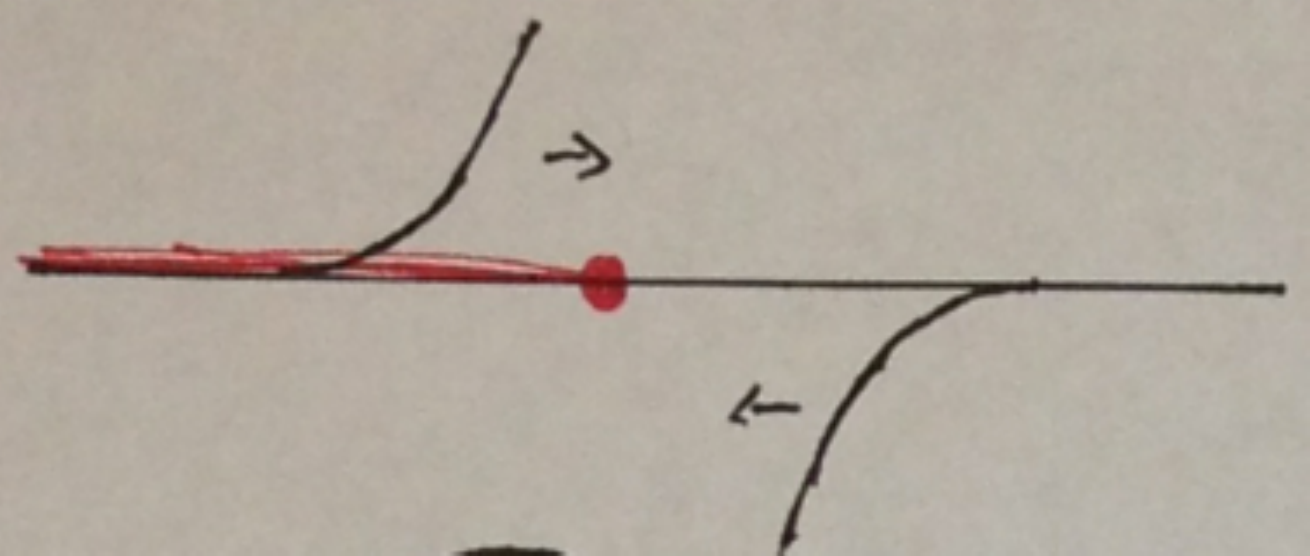}\\\includegraphics[width=140pt]{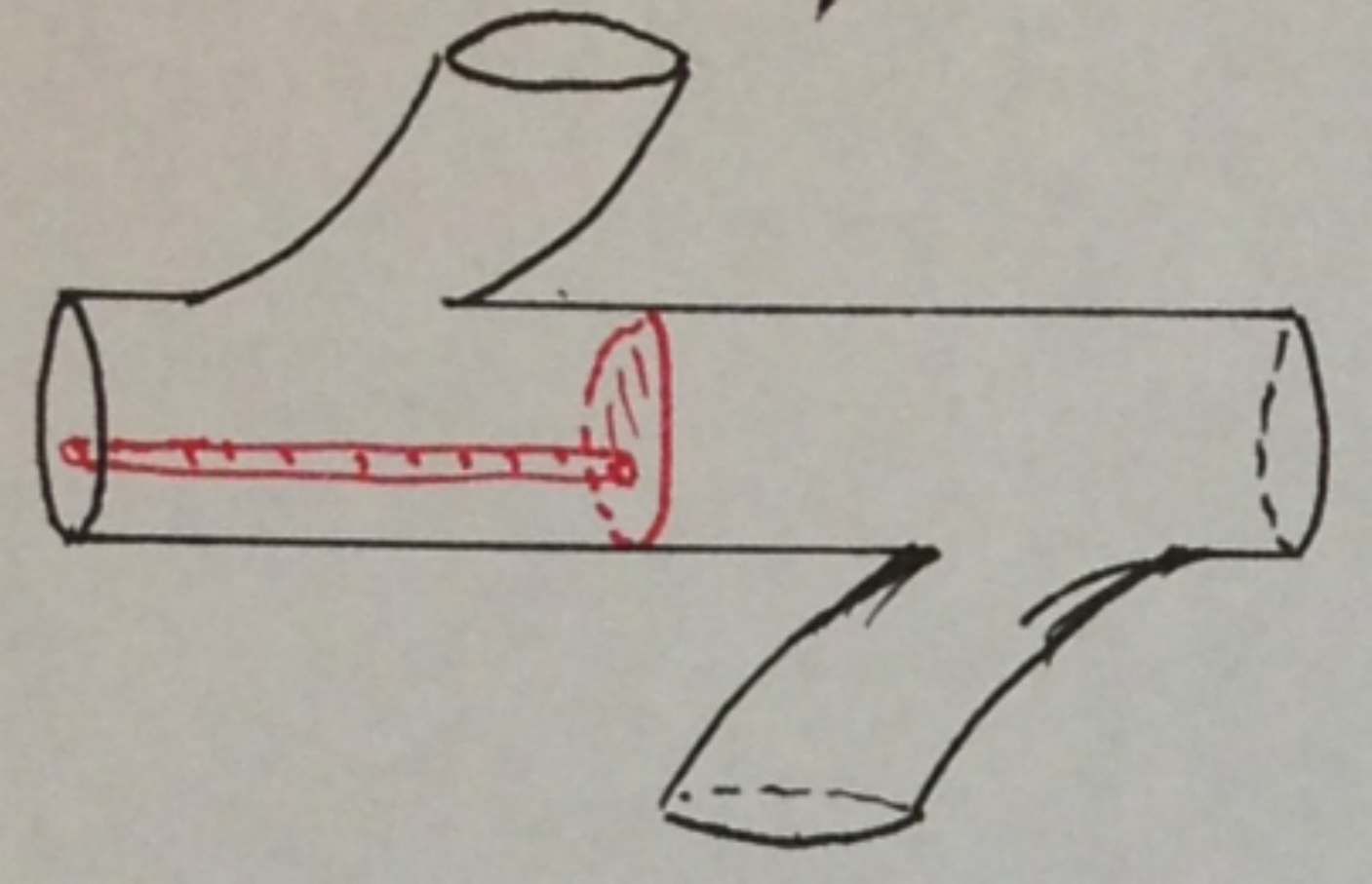}}}
&$\buildrel exch. \over \longrightarrow$&
\fbox{\minibox{\includegraphics[width=140pt]{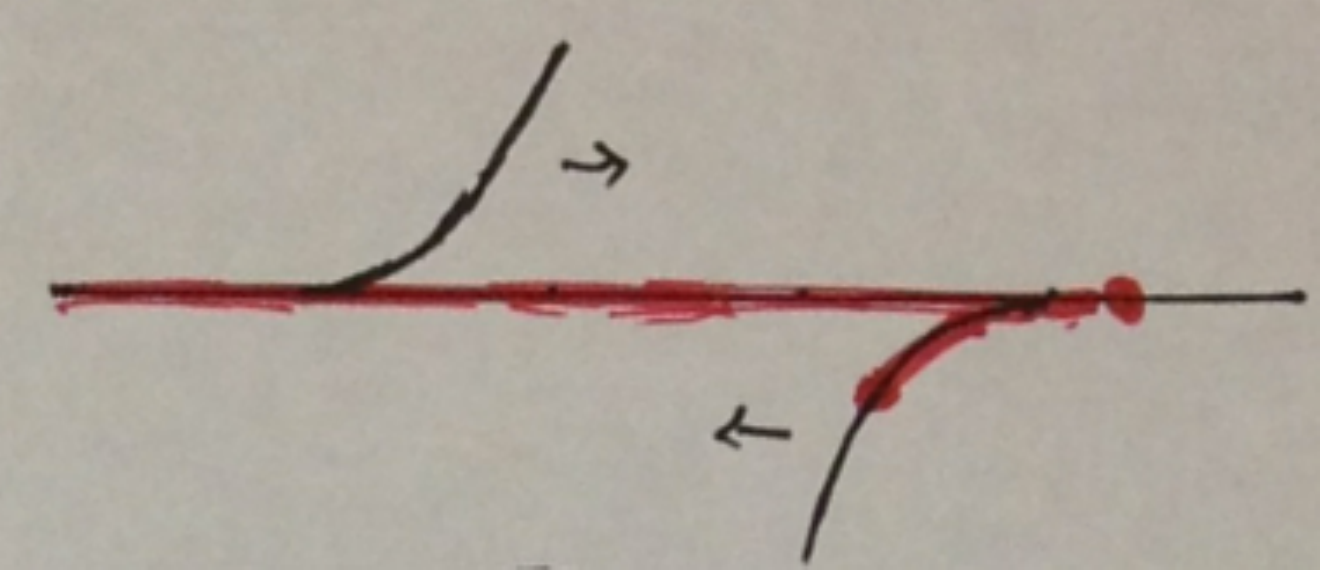}\\\includegraphics[width=140pt]{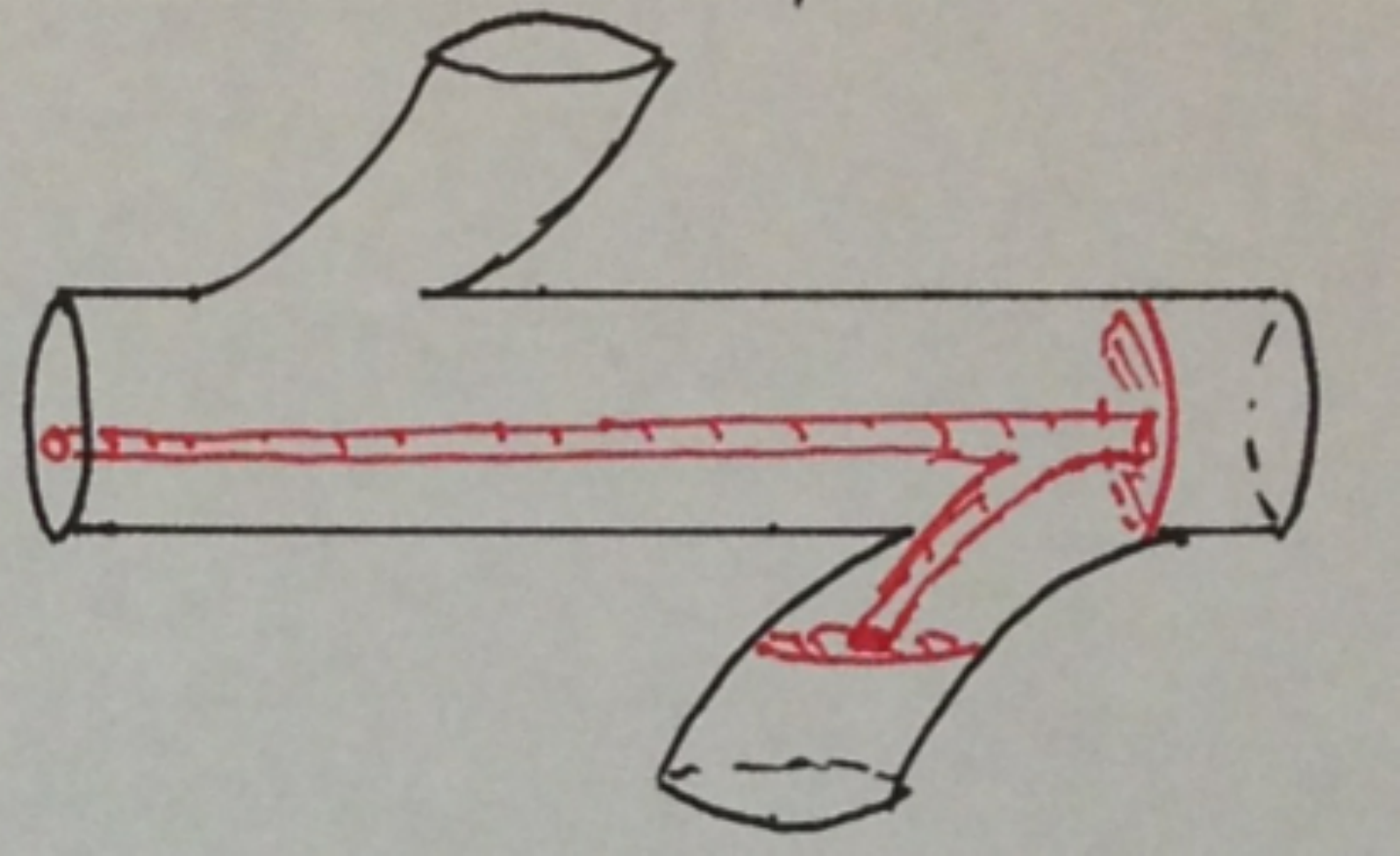}}}
\\\vspace{0.4cm}\\$\buildrel fold \over \longrightarrow$&
\fbox{\minibox{\includegraphics[width=140pt]{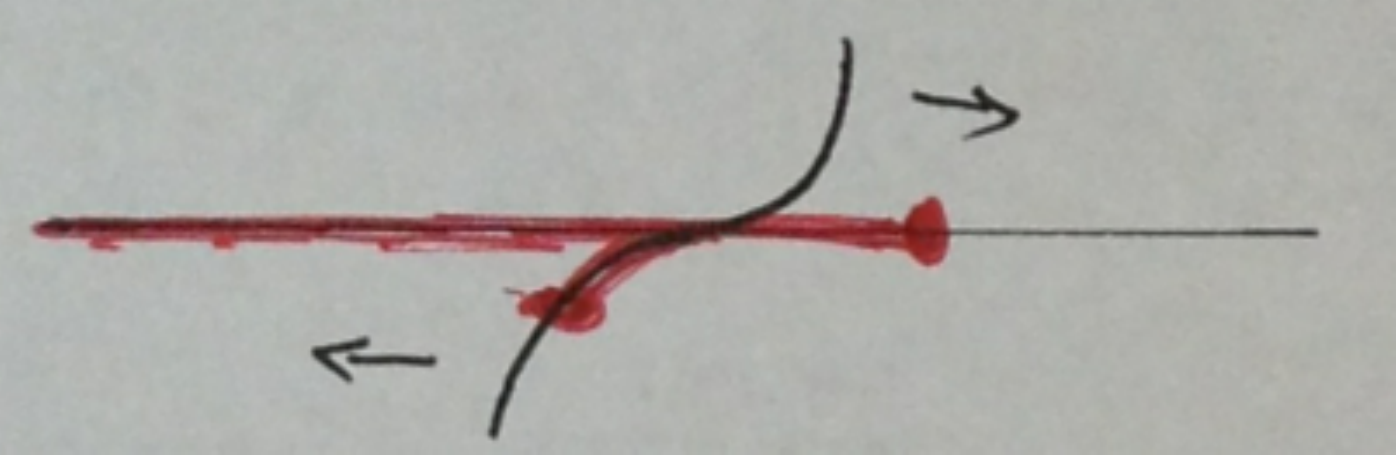}\\\includegraphics[width=140pt]{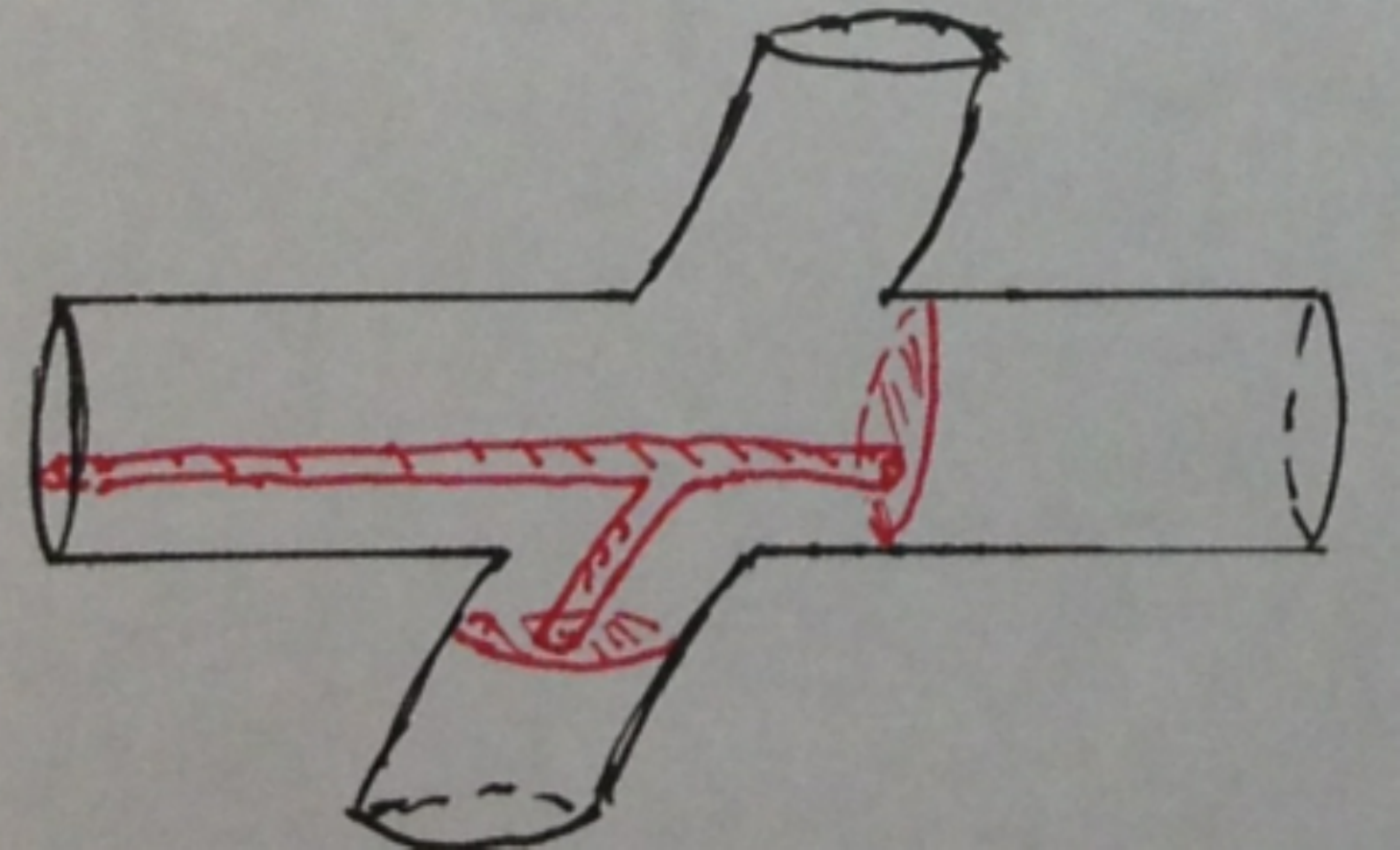}}}
&$\buildrel fold \over \longrightarrow$&
\fbox{\minibox{\includegraphics[width=140pt]{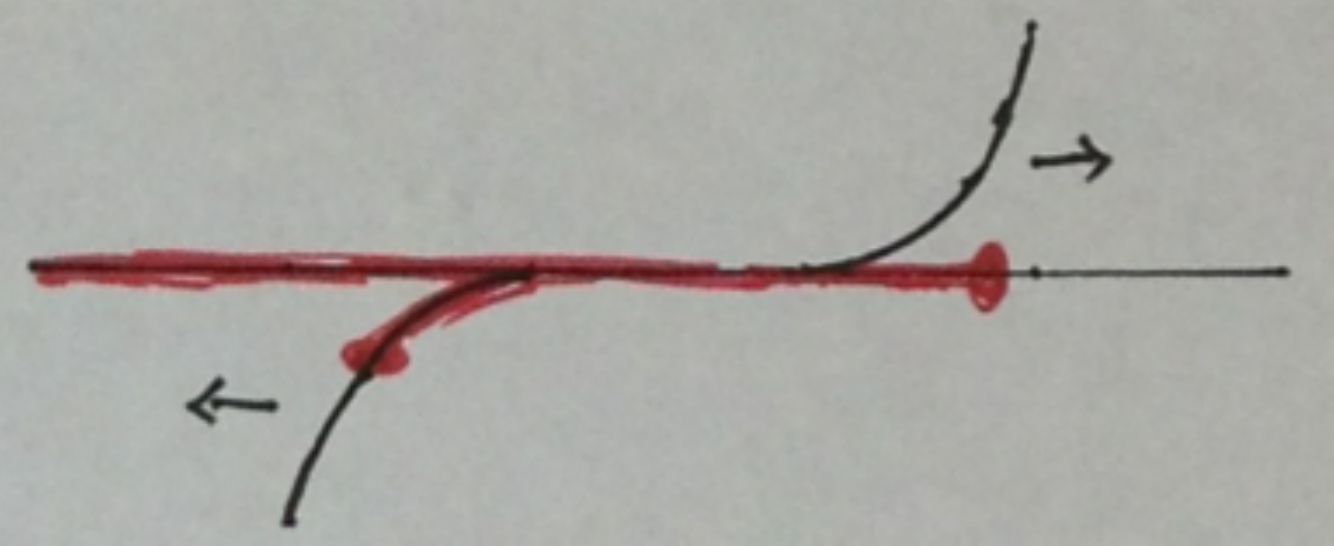}\\\includegraphics[width=140pt]{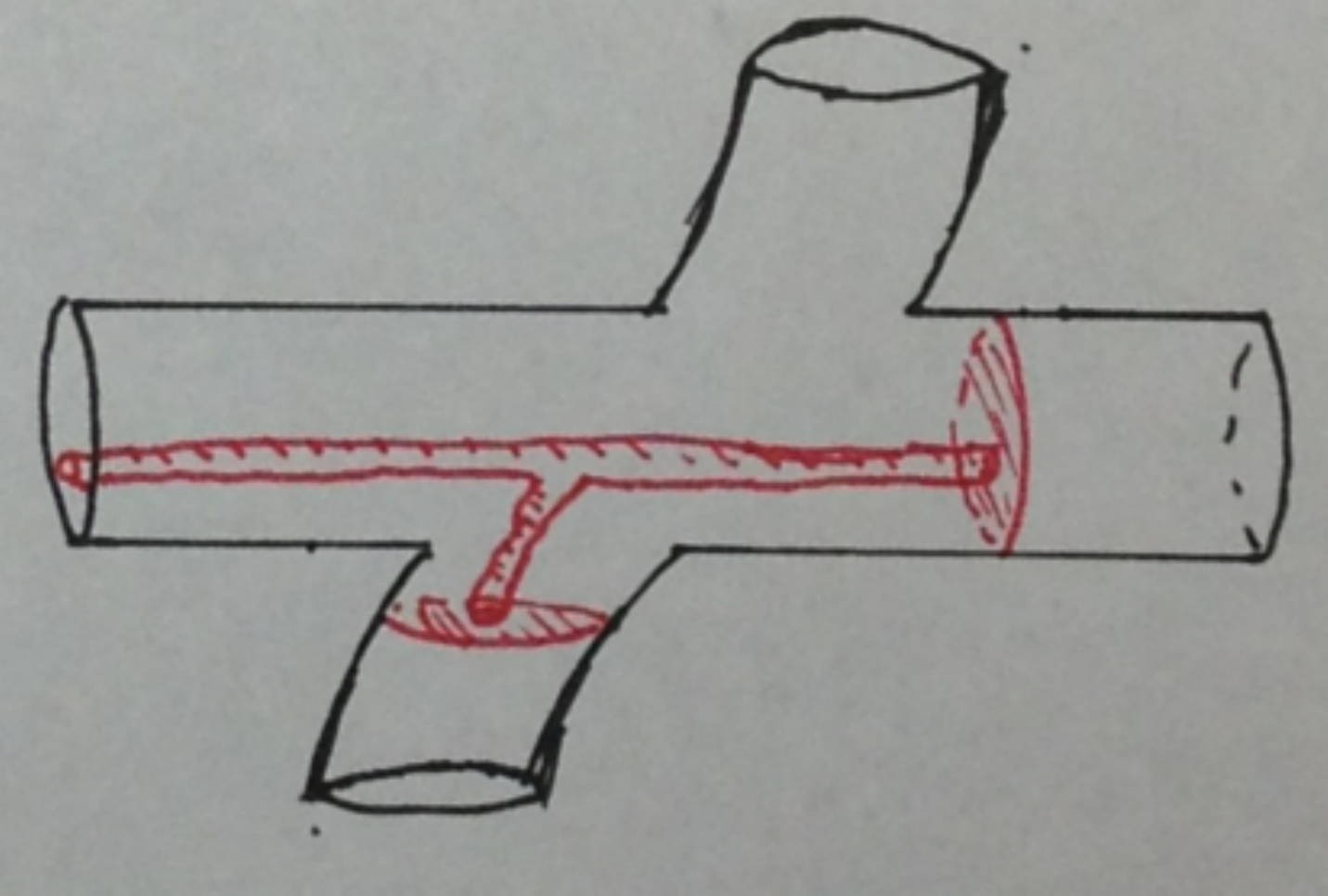}}}
\\
\end{tabular}
\caption{Sphere evolution along a folding path.  Each box shows a portion of the sphere tree with the gates indicated on top and the corresponding sphere in the doubled handlebody (with only one half of the handlebody drawn) on bottom.  As the gates are folded at unit speed, the bud in the first box is exchanged for the buds in the second box, and then the gates are folded normally past the gate collision in the third box.}\label{fig:sphere_evolution}
\end{center}
\end{figure}\todo{redraw}



\section{Sphere Trees as Slices of the Guirardel Core}\label{sec:core}

In this section we point out a useful connection between sphere trees and the Guirardel core via work of Behrstock, Bestvina, and Clay~\cite{behrstock_bc:intersection_numbers10}.  This section is not necessary for the remainder of the paper, so the reader interested only in projection and not in properties of sphere trees can skip it.

It turns out that sphere trees appear in the literature in a different guise.  Sphere trees are closely related to the trees studied by Behrstock, Bestvina, and Clay as slices of the Guirardel core.  The Guirardel core is a way of assigning a canonical CAT(0) geometry to a pair of splittings for groups acting on trees.  It is extremely useful in the study of $Out(F_n)$ and elsewhere, yielding for instance an intersection number between two points in outer space \cite{Guirardel}.  However, as the definition requires some exposition and the core is only tangential to our purposes, we defer its definition and a discussion of the core to the references.  Instead, in this section we focus on the work of Behrstock, Bestvina, and Clay.

Let $T_0 \in \cv$ be an arbitrary marked metric tree with quotient $\Gamma = F_n \backslash T_0$ having fundamental group $F_n$.  Let $e$ denote an edge of $T_0$ which covers a nonseparating edge in $\Gamma$.  These authors describe a construction for finding slices of the Guirardel core $\core(T_0, T)$ that lie above a point $p_e$ on the edge $e$ in $T_0$.  As the core is a subset of $T_0 \times T$, such a slice is a subtree of $T$, which is defined as follows.  Let $\phi \co T \to T_0$ be a map which is equivariant with respect to the actions of $F_n$ on the vertices of the trees and which is injective when restricted to the edges of $T$.  Elements in $\phi^{-1}(p_e)$ are called \emph{buds} (this set was denoted by $\tilde{\Sigma}_e$ in \cite{behrstock_bc:intersection_numbers10}).  Because the map $\phi$ is injective on edges, there is at most one bud per edge of $T$.  We call the convex hull $Y$ of the set of buds (together with the set of buds itself) a \emph{core tree} of $e$ with respect to $T$.

A vertex $v$ of the core tree is \emph{removable} if all but one direction from $v$ is towards an adjacent edge containing a bud on the boundary of the core tree.  If $\phi$ is such that there are no removable vertices, $\phi$ is called \emph{consolidated}.  Behrstock, Bestvina and Clay prove that $\phi$ may be equivariantly perturbed to remove all removable vertices via a process they call \emph{pruning}.  Once we show that core trees and sphere trees are the same object, it will become apparent that the pruning move they define is a special case of the Exchange Move above (specifically, it is the second step of the Sphere Tree Simplification Algorithm).

A core tree $Y$, being the convex hull of a set of non-vertex points, may contain some full edges of $T$ but always contains portions of other edges of $T$.  Let $core(Y)$ denote the subtree of $Y$ consisting only of those full edges of $T$ that are entirely contained in $Y$.

\begin{theorem}\cite{behrstock_bc:intersection_numbers10}
The tree $\emph{core}(Y)$ of a consolidated core tree $Y$ is the slice of the Guirardel core $\core(T, T_0)$ lying above a point $p_e$ on the edge $e$.
\end{theorem}

We now show that sphere trees and these core trees are in fact the same concept.

\begin{theorem}
Let $\mA$ be any simple sphere system in $M_n$ with associated dual graph $\Gamma$, and let $T := \tilde\Gamma$.  Let $S$ denote an essential embedded sphere in $M_n$ and let $T_0$ denote the universal cover of the dual graph in $M_n$ associated to any simple sphere system containing $S$.  There exists a consolidated sphere tree for $S$ with respect to $\mA$ that is a core tree for the edge associated to $S$ with respect to $T_0$.
\end{theorem}

\begin{proof}

Let $\phi \co T \to T_0$ be an optimal map.  Consider the sphere tree for $S$ with respect to $T_0$.  This is just a sphere tree with a single point $p$ that is a bud on an edge of $T_0$ corresponding to $S$.  To construct a sphere tree for $S$ with respect to $T$, we reverse the evolution of this point $p$ along the folding path for $\phi$.  The result must be a sphere tree $T_S$ for $S$ in $T$ that folds (via $\phi$) to the point $p$.  Under this reverse evolution, we will always maintain that the intermediate sphere trees are consolidated.  It is easy to verify that reversing sphere tree evolution (specifically, the Bud Exchange move that evolution relies on) is such that the buds for $T_S$ may be chosen to be the set of preimages of $p$.  This shows that there exists a consolidated sphere tree that coincides with a consolidated core tree.
\end{proof}

\begin{corollary}
The tree $\emph{core}(T_S)$ of a consolidated sphere tree $T_S$ is the slice of the Guirardel core $\core(T, T_0)$ lying above a point on the edge of $T$ corresponding to $S$.
\end{corollary}

\begin{proof}
This corollary holds for the consolidated sphere tree $T_{core}$ identified in the previous theorem as coinciding with a core tree.  Any other consolidated sphere tree for $S$ differs from $T_{core}$ by a sequence of Bud Exchanges and Bud Cancellations, which correspond to applying different choices of these moves while unfolding the map $\phi$ in the proof of the previous theorem.  When the property of being consolidated is maintained note that Bud Exchange and Bud Cancellation preserve the $core$ subgraph of the sphere trees.
\end{proof}

Recall the procedure for evolution of sphere trees allows us to algorithmically construct a sphere tree $T_S$ for a given sphere $S$ with respect to a given tree $T\in\cv$ as follows.
    \begin{enumerate}
    \item Choose a tree $Y \in \cv$ compatible with $S$ in the sense that a lift of $S$ is parallel to the sphere lying over the midpoint of some edge $e$ of $Y$ in the fibration of $\tilde M_n$ with base space $Y$ (or, in other words, $Y$ corresponds to a splitting of $F_n$ which is a refinement of the splitting of $F_n$ induced by $S$).  Choose an optimal map $\phi \co Y \to T$ inducing a train track structure on $Y$.  Such a map always exists (this follows for instance from Proposition 2.5 of~\cite{bestvina_f:hyperbolicity_of_ff11}).
    \item Let $Y_S$ denote a sphere tree for $S$ with respect to $Y$. Since $Y$ was chosen to be compatible with $S$, $Y_S$ consists of a single bud on the midpoint of any lift of $e$ and no twigs.
    \item Evolve $Y_S$ along the folding path associated to $\phi$.
    \end{enumerate}
By the above theorems, this procedure also allows us to algorithmically construct Behrstock-Bestvina-Clay core trees as well as slices of the Guirardel core by using folding paths.

\section{Bounded Geodesic Image theorem}\label{sec:maintheorem}

We are now ready to prove our main theorem, the Bounded Geodesic Image theorem.


\begin{theorem}[Bounded Geodesic Image]\label{thm:boundedgeodesicimage}  Let $S \subset Y\subset M_n$ be an essential nonseparating embedded sphere in a submanifold $Y$ of the doubled handlebody such that $Y$ exhausts $M_n$ and $\SC(Y)$ is hyperbolic.  Let $X := Y - S$.  For any geodesic segment, ray or line $\gamma$ in $\SC(Y)$ such that $\gamma$ does not contain $[S]$, the set $\pi_X(\gamma)$ has uniformly bounded diameter in $\SC(X)$.
\end{theorem}

This theorem should be compared to the recent version of the Bounded Geodesic Image theorem of Bestvina and Feighn \cite{bestvina_f:subfactor_projection12}, where in order to get bounded diameter of the projection of geodesic, this geodesic has to avoid the 4-neighborhood of the vertex $[S]$.

\begin{proof}
The case when $\gamma$ is a geodesic line follows easily from the case when $\gamma$ is a finite geodesic path.  Let $[A]$ and $[B]$ denote the endpoints of $\gamma$, and let $\gamma_A$ and $\gamma_B$ be geodesics from $[A]$ to $[S]$ and from $[B]$ to $[S]$, respectively. Since $\SC(Y)$ is hyperbolic, this is a uniformly thin triangle.  Thus, it suffices to prove that $\pi_X(\gamma_A-\{[S]\})$ and $\pi_X(\gamma_B-\{[S]\})$ are uniformly bounded, since $\pi_X$ is Lipschitz (note here we use that $\gamma$ does not contain $[S]$, and that there is at least one point on $\gamma$ that are within distance $\delta + 1$ from \emph{both} $\gamma_A$ and $\gamma_B$, where $\delta$ is the hyperbolicity constant).  Moreover, we only need to show projection is bounded for $\gamma_A-\{[S]\}$, as the argument for $\gamma_B - \{[S]\}$ will be identical.

Consider the geodesic $\gamma_A$.  Let $\hat A$ and $\hat S$ denote reduced sphere systems containing $A$ and $S$, so that, in particular, the corresponding trees $T_{\hat A}$ and $T_{\hat S}$ in the outer space are locally finite. Fix some terse map from $T_{\hat A}$ to $T_{\hat S}$ and let $(T_t)$, $t\in[0,t_{last}]$ be the corresponding terse folding path in the outer space between them. This folding path and what follows below is shown in Figure~\ref{fig:folding_path}.  The folding path projects to a sequence $\{\hat S_i\}_{0\leq i\leq L}$ of simplices in $\SC(Y)$ with $\hat S_0=\hat A$ and $\hat S_L=\hat S$ that is an unparametrized quasigeodesic according to Lemma~\ref{lem:terse_paths_project_to_geodesics}. We can assume that in this sequence each simplex is either a face or coface of each of its neighbors. Thus, to prove $\pi_X(\gamma_A)$ is uniformly bounded, it suffices to find spheres $S_i$ for $i = 0, \dots, L-1$ so that for each $i$, $[S_i]$ is uniformly close to $[\hat S_i]$ in $\SC(Y)$, and $\pi_X(\{[S_i]\}_{i=0}^{L-1})$ is uniformly bounded. Below, we explicitly construct such spheres.  From now on we will repeatedly use the fact that $(T_t)$ is a folding path, and we no longer need terseness (which is only used so show quasigeodicity), so from now on we refer to $(T_t)$ as a folding path.

\begin{figure}[h]
\begin{center}
\includegraphics[width=\textwidth]{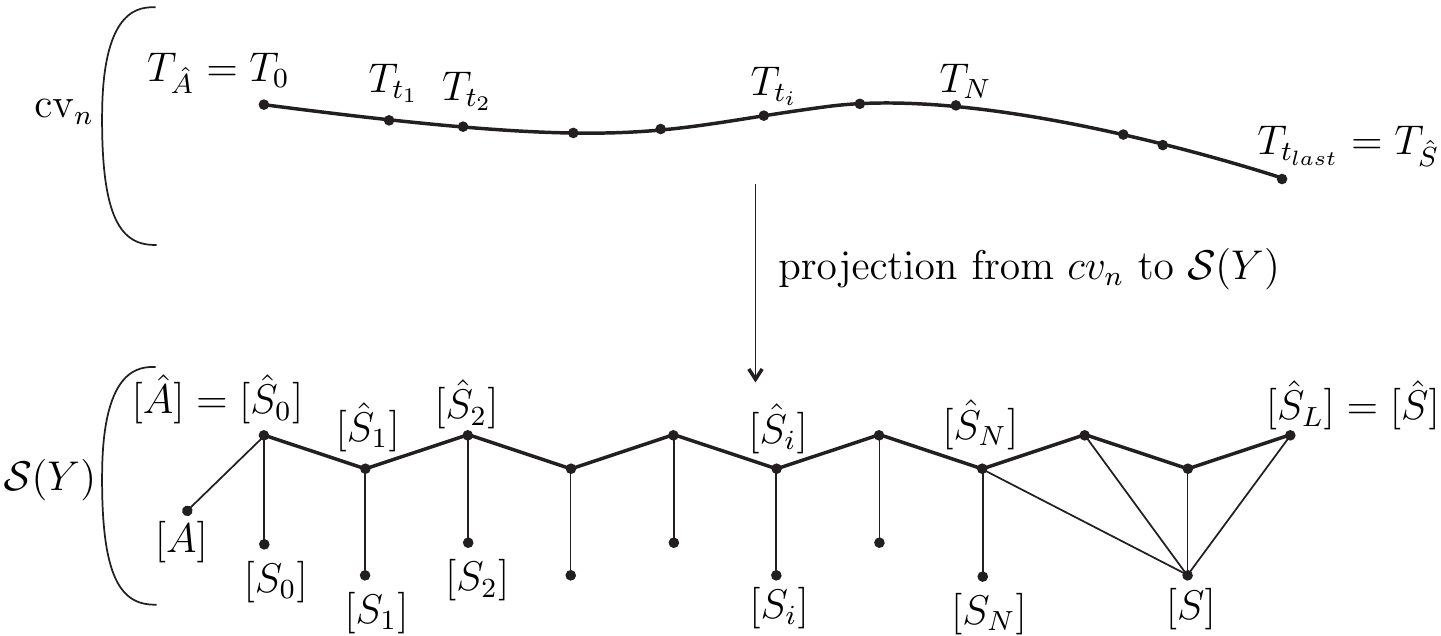}
\caption{Projection of the folding path $(T_t)$,$t\in[0,t_{last}]$ to $\SC(Y)$\label{fig:folding_path}}
\end{center}
\end{figure}

For all $0 \leq i \leq L$ fix $t_i \in [0,t_{last}]$ for which $T_{t_i}=\hat S_i$.

The graph of groups decomposition of the splitting of $F_n$ corresponding to the sphere $S$ has underlying graph $\Gamma_S$ containing exactly one edge $e$.  The graph $\Gamma_{\hat S} := F_n \backslash T_{\hat S}$, given trivial edge and vertex labels, is the graph of groups decomposition of $F_n$ with respect to the sphere system $\hat S$.  Moreover, the fact that $S \subset \hat S$ induces a map between graph of groups decompositions, which has the effect of effect of sending $\Gamma_{\hat S}$ to $\Gamma_S$ by collapsing all but one edge to a point. We denote by $\hat e$ the non-collapsed edge of $\Gamma_{\hat S}$, and by $\tilde e$ the complete preimage of $\hat e$ in $T_{\hat S}$.

We can trace back in time to the last time before the edge $\hat e$ appears in any quotient graph along the folding path $T_t$.  To be precise, we define the `last time before $\hat e$ appears' as the last time $N$ such that there does not exist a point in the interior of $\hat e$ which has a single preimage in the graph $F_n \backslash T_N$.  If $\hat e$ is present in $T_0=T_{\hat A}$, then $\hat S$ and $\hat A$ share a common sphere, so $S$ is distance at most $2$ from $A$, and hence $\pi_X(\gamma_A-\{[S]\})$ is bounded by Proposition \ref{prop:properties}.  It also follows that for each $t\in (N, t_{last}]$ the sphere system corresponding to $T_t$ contains $S$ and thus the projection of $T_t$ to the sphere complex will be to a simplex that has $[S]$ as a face.  The projection of $T_{N}$ to the sphere complex belongs to the link of $[S]$ in $\SC(Y)$.  Thus, as we are trying to prove that the projection of a geodesic beginning at $[S]$ is bounded and such a geodesic contains at most one point in the link of $[S]$, we need not consider the projection of the folding path past time $N$.

We wish to choose the spheres $S_i$ mentioned above, but we must be careful with our choices.  Intuitively, we make these choices by `pulling back' to time $0$ a gate at time $\tau_N$ that creates the edge $\hat e$, and then using this preimage gate to define the $S_i$.  We now formally describe what we mean by this last sentence.

The time $N$ was chosen so that the edge $\hat e$ appears immediately after time $N$.  Thus, there exists some gate $\tau_N$ at a vertex $v_N$ in $\Gamma_N := F_n \backslash T_N$ that contains at least two directions and, after being folded, creates the edge $\hat e$ behind it.  Let $\hat\tau_N$ be an arbitrary lift of $\tau_N$ in $T_N$. Fix any two directions $d_N$ and $d'_N$ in $\hat\tau_N$.  For each $0 \leq t \leq N$ we choose a gate $\hat\tau_t$ with vertex $v_t$ as well as directions $d_t$ and $d'_t$ in $\hat\tau_t$ that evolve under folding to $\hat\tau_N$, $v_N$, $d_N$, and $d'_N$, respectively, and which are chosen as follows.

Fix points $p_N$ and $p'_N$ close to $v_N$ on the edges of $T_N$ in the directions $d_N$ and $d'_N$, respectively.  Let $f: T_0 \to T_N$ be the optimal map corresponding to the folding path.   Let $p_0\in f^{-1}(\{p_N\})$ and $p'_0\in f^{-1}(\{p'_N\})$ be preimages of the points $p_N$ and $p'_N$, and consider the unique geodesic path in $T_0$ from $p_0$ to $p'_0$. This path may contain other points from $f^{-1}(\{p_N,p'_N\})$, but since it starts in $f^{-1}(\{p_N\})$ and ends in $f^{-1}(\{p'_N\})$, there must be some subpath $\gamma_0$ that starts in $f^{-1}(\{p_N\})$, ends in $f^{-1}(\{p'_N\})$, and also does not contain any other points from $f^{-1}(\{p_N,p'_N\})$.  Without loss of generality, assume $p_0$ and $p'_0$ are chosen so that they are the endpoints of $\gamma_0$.  For each $t \in (0,N]$, let $p_t$ and $p'_t$ be the points in $T_t$ which are the images along the folding path of $p_0$ and $p'_0$, respectively, and let $\gamma_t$ be the geodesic path between $p_t$ and $p'_t$.

By construction, after folding $T_0$ to $T_N$, the path $\gamma_0$ maps to a path connecting $p_N$ and $p'_N$ and so contains $\gamma_N$ and hence $v_N$.  Moreover, since there are no points in the interior of $\gamma_0$ which fold to either $p_N$ or $p'_N$, no illegal turn contained in any $\gamma_t$ can fold past the the endpoints of $\gamma_t$ as $\gamma_t$ evolves along the folding path. Therefore, along the folding path, each illegal turn in each $\gamma_t$ will either stop being illegal after some time or will continuously  evolve and persist until time $N$. The process of folding of $\gamma_0$ to $\gamma_N$ is shown in Figure~\ref{fig:gate_preimage}. Since $\{d_N,d'_N\}$ is the only illegal turn contained in $\gamma_N$, and since legal turns stay legal (hence illegal turns stay illegal when time flows backwards), there must be at least one turn contained in $\gamma_0$ that evolves to the turn $\{d_N,d'_N\}\subset \hat\tau_N$. Arbitrarily choose some such turn $\{d_0,d'_0\}$, where $d_0$ and $d'_0$ are directions in $T_0$ that point towards $p_0$ and $p'_0$, respectively.  Let $\hat\tau_0$ denote the gate of $T_0$ containing $\{d_0,d'_0\}$ and let $v_0$ denote the corresponding vertex. In Figure~\ref{fig:gate_preimage} the candidates for the vertex $v_0$ are labeled with dots. The vertex $v_0$ evolves to some vertex $v_t$ along $\gamma_t$.  Let $d_t$ and $d'_t$ denote the directions from $v_t$ which point towards $p_t$ and $p'_t$, respectively.  The turn $\{d_t, d'_t\}$ must be illegal, by the choice of $\{d_0,d'_0\}$, and so there is a gate $\hat\tau_t$ in $T_t$ containing $\{d_t,d'_t\}$.  Let $v_t$ be the vertex corresponding to $\hat\tau_t$.

\begin{figure}[h]
\begin{center}
\includegraphics{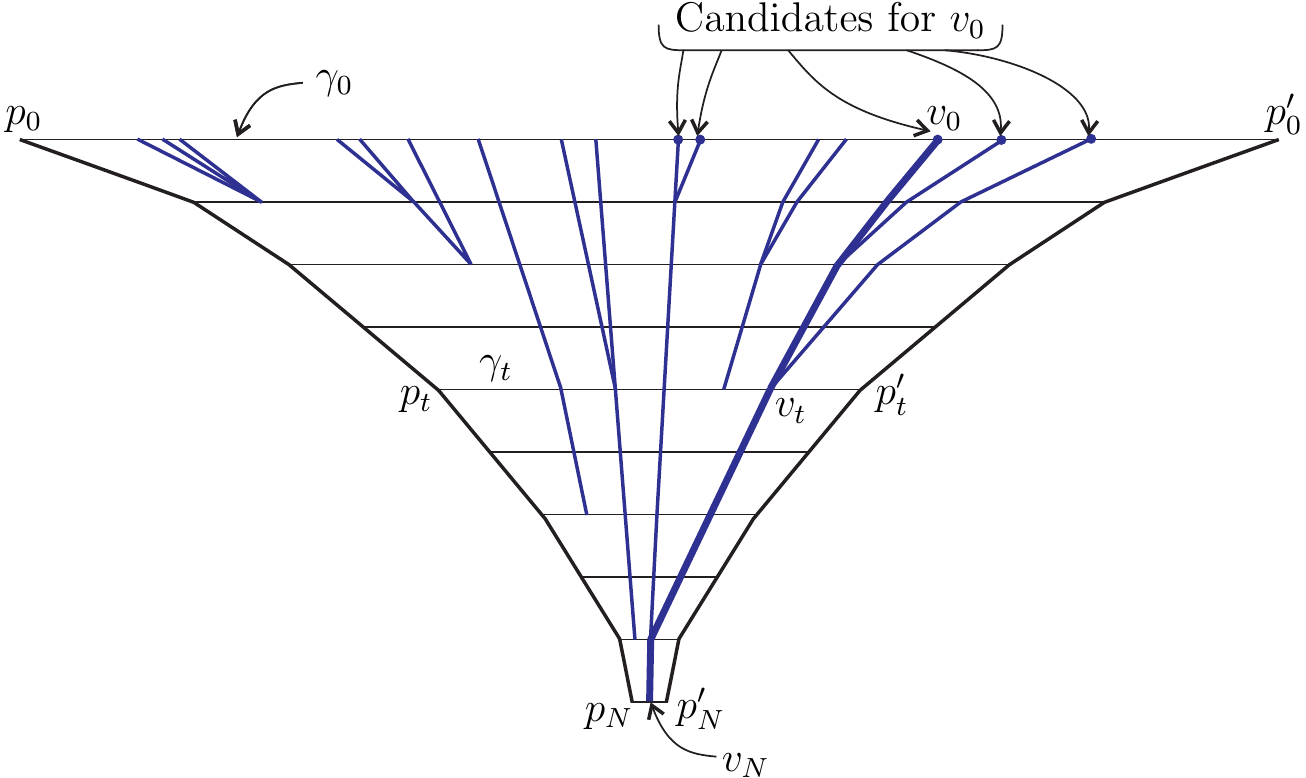}
\caption{Selecting of the preimage $\tau_0$ of the gate $\tau_N$\label{fig:gate_preimage}}
\end{center}
\end{figure}

We are now ready to define the spheres $S_i$.  At time $t_i$, $i=0,\ldots,N$ the vertex at the gate $\hat\tau_{t_i}$ must have at least one other gate besides $\hat\tau_{t_i}$. Pick an arbitrary direction $d''_{t_i} \not\in \hat\tau_{t_i}$ at $v_{t_i}$.  Define the sphere $S_i$ to be the sphere whose sphere tree consists of exactly two buds, one on each of the two edges adjacent to $v_{t_i}$ in the directions $d_{t_i}$ and $d''_{t_i}$.  This sphere tree for $S_i$ is depicted in Figure~\ref{fig:def_of_S_i}. Then by construction $S_i$ is disjoint from (though not contained in) the sphere system $\hat S_i$. Therefore, to finish the proof, it is enough to show now that the projections of all $S_i$ to the link $\SC(X)$ of the vertex $[S]$ coincide.

\begin{figure}[h]
\begin{center}
\includegraphics{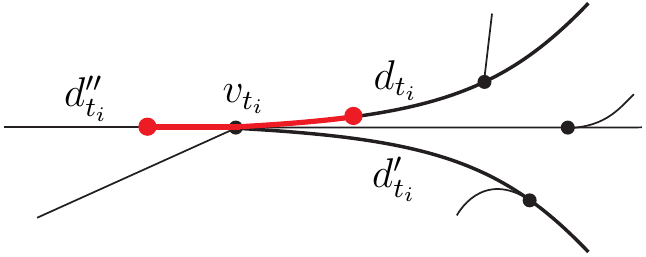}
\caption{Definition of a sphere $S_i$\label{fig:def_of_S_i}}
\end{center}
\end{figure}

One obtains the consolidated sphere tree corresponding to $S_i$ with respect to $\hat S_N$ by folding the tree $T_{t_i}$ along the folding path and following the rules of evolving sphere trees in Section \ref{sec:evolving}. Since along the folding path all illegal turns are folded at the unit speed, other gates can merge with $\hat\tau_t$, but no other gate can fold past the gate $\hat\tau_{t}$ from any direction which is not in $\hat\tau_t$. On the other hand, even if some gate folds past $\hat\tau_{t}$ from a direction which is in $\hat\tau_{t}$, then according to the rules of folding sphere trees it will still be the case that exactly one direction ($d_t$) of $\hat\tau_t$ is not an external direction for the sphere tree, and there is an end bud in the direction $d_t$.  Therefore, at time $N$ this statement still holds:  \textbf{the consolidated sphere tree for $S_i$ viewed in $T_N$ must contain an end bud in the direction $d_N$ and no other buds in any direction contained in $\hat\tau_N$ from $v_N$}.  Note we have no control over what happens (aside from what is prescribed by the rules of folding sphere trees) in the direction $d''_t$, as the rest of the sphere tree for $S_i$ must be in this direction.

\begin{figure}[h]
\begin{center}
\begin{tabular}{c}
\includegraphics{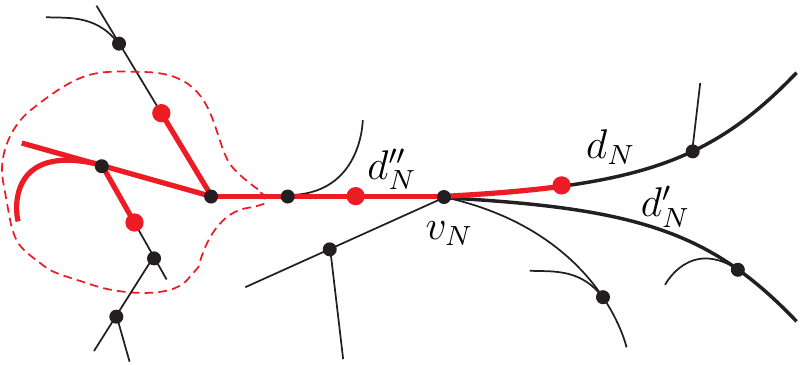}\\
~\hspace{4cm}{\huge$\downarrow$} \textit{(fold for time $\epsilon$)}\\[.1cm]
\includegraphics{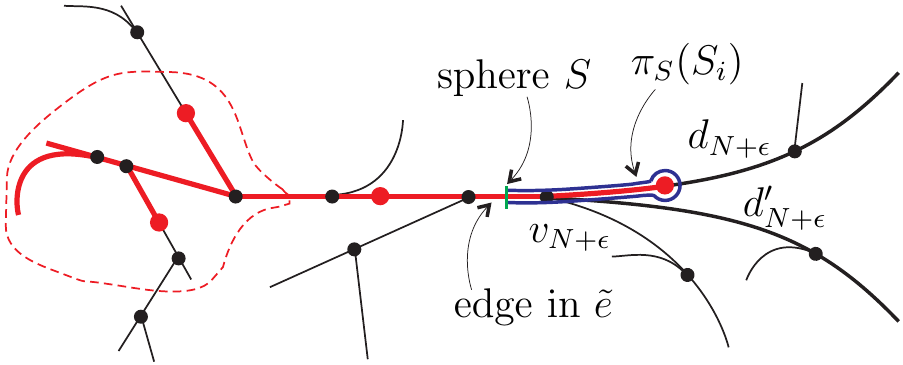}\\
\end{tabular}
\caption{Computing the projection $\pi_S(S_i)$ (Case I)\label{fig:sphere_tree_N1}}
\end{center}
\end{figure}

After folding $\tau_N$ for a small time $\epsilon$ further (so that no other vertex collisions happen and no illegal turn stops being folded) and creating the edge $\hat e$ in $F_n\setminus T_{N+\epsilon}$, the projections of the $S_i$ to $\SC(X)$ have sphere trees which are obtained from the consolidated sphere tree for $S_i$ in $T_{N+\epsilon}$ by chopping along the midpoint of the edges in $\tilde e$.

\begin{figure}[h]
\begin{center}
\begin{tabular}{l}
\includegraphics[width=3in]{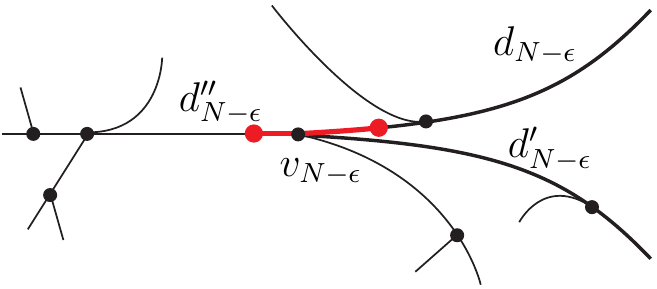}\\
~\hspace{3.2cm}{\huge$\downarrow$} \textit{(fold for time $\epsilon$)}\\[.25cm]
\includegraphics[width=3in]{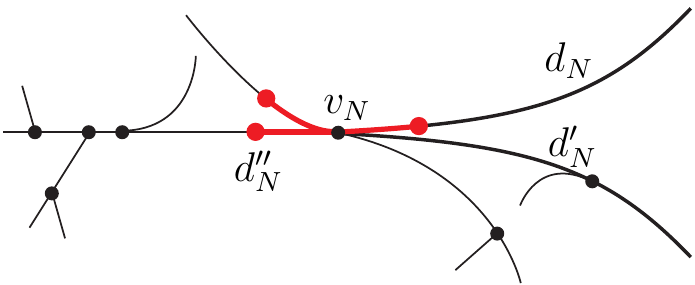}\\
~\hspace{3.2cm}{\huge$\downarrow$} \textit{(fold for time $\epsilon$)}\\[.25cm]
\includegraphics[width=3in]{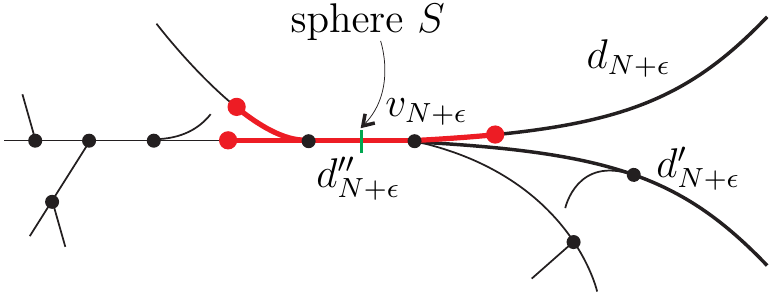}\\
~\hspace{3.2cm}{\huge$\downarrow$} \textit{(Bud Exchange move)}\\[.25cm]
\includegraphics[width=3in]{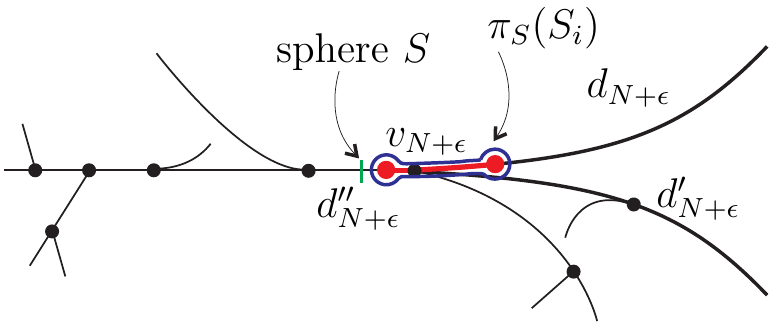}\\
\end{tabular}
\caption{Computing the projection $\pi_S(S_i)$ (Case II)\label{fig:sphere_tree_N2}}
\end{center}
\end{figure}

Since $N$ was chosen to be last time we do not see the edge $\hat e$ in $\Gamma_t$, there is at least one more direction at $v_N$ that is not in the gate $\hat\tau_N$ (otherwise the edge $\hat e$ should have been introduced earlier). Further, since we picked $\epsilon$ in such a way that no vertex collision happens and no illegal turn stops being illegal from time $N$ to time $N+\epsilon$, at time $N+\epsilon$, the vertex $v_{N+\epsilon}$ has precisely two gates, one of which is $\tau_{N+\epsilon}$ and the other one corresponds to a newly created edge in $\tilde e$.

By construction, there are two possibilities for the sphere tree of $S_i$ in $T_{N+\epsilon}$: either it will contain the whole edge adjacent to $v_{N+\epsilon}$ from the orbit $\tilde e$, in which case its projection to $\SC(X)$ is a relative sphere with volume $0$ and a single bud in direction $d_{N+\epsilon}$ (see Figure~\ref{fig:sphere_tree_N1}); or, via Bud Exchange moves, it will be a connected sum of a sphere $S$ and a sphere in the direction $d_{N+\epsilon}$ and will coincide with its projection to $\SC(X)$ (see Figure~\ref{fig:sphere_tree_N2}).

As the projections of all of the $S_i$ are either disjoint, or distance 2 apart, the theorem is proven.

\end{proof}

\begin{corollary}\label{cor:largedistance}
Given the notation of the Bounded Geodesic Image theorem, for any spheres $A, B \subset Y$, if $d_X([A],[B])$ is large then $[S]$ is on every geodesic in $\SC(Y)$ connecting $[A]$ and $[B]$.
\end{corollary}

\section{Future Directions}\label{sec:futuredirections}

There are many interesting directions and questions arising from this work.

Most pressing is the relationship between our definition of projection and that of Bestvina and Feighn \cite{bestvina_f:subfactor_projection12}.  Do these notions of projection coincide, or do they at least differ by a uniformly bounded amount?  In $\tilde M_n$, the universal cover of the doubled handlebody, the lift of a given sphere system and its associated Bass-Serre tree are dual to each other.  Does this duality lead to a relationship between submanifold projection and subfactor projection?

Bestvina and Feighn also use their projection to verify the axioms of Bestvina, Bromberg, and Fujiwara \cite{BestvinaBrombergFujiwara} to show that $\Out$ acts on a finite product of hyperbolic spaces.  The work contained here proves results similar to these axioms, indicating that this may lead to an independent proof of Bestvina and Feighn's result.

Our definition of projection makes sense for projecting isotopy classes of spheres to sphere systems.  There are various complexes which are quasiisometric to the factor complex and that have as vertices the isotopy classes of spheres.  Many of our results for projection hold for such complexes, but it is not clear that the projection map in this context is still Lipschitz in the sense of Proposition \ref{prop:properties}.  Does submanifold projection satisfy other desirable properties for the factor complex?

For our proof of Bounded Geodesic Image theorem, we needed to use special properties of folding paths.  We know that fold paths are quasigeodesics in the sphere complex \cite{handel_m:hyperbolicity_of_fs12} and in the factor complex \cite{kapovich_r:hyperbolicity_FF12}, and that folding paths are quasigeodesics in the factor complex \cite{bestvina_f:hyperbolicity_of_ff11}.  We suspect that projections of folding paths to the sphere complex are also quasigeodesics, which if true would loosen our restriction of requiring terse paths.  Note that surgery paths also form a coarsely transitive path family of quasigeodesics in the sphere complex \cite{hilion_h:hyperbolicity_FS12}.

An interesting notion to explore is that of the volume of a sphere tree, defined as the metric volume of the underlying core.  The volume of a sphere tree appears to behave nicely under folding paths.  We suspect that the projection of a vertex to a geodesic used to prove hyperbolicity \cite{bestvina_f:hyperbolicity_of_ff11,handel_m:hyperbolicity_of_fs12} in either the case of the splitting complex or the factor complex is uniformly boundedly close to the place along the geodesic where the associated sphere tree has minimal volume.

One application of the Bounded Geodesic Image theorem could be to provide a simple way of showing that the sphere complex and the factor complex are not quasiisometric, as follows.  Every reducible element of $\Out$ acts with bounded orbit on the factor complex.  Does there exist a reducible element that acts with unbounded orbit on the sphere complex?  This could be shown by providing a reducible element and an orbit of points in the sphere complex such that the projection of the orbit to each point in the orbit has large diameter.  This would show via Corollary \ref{cor:largedistance} that the orbit lies on a geodesic line.

Of course, the ultimate goal in defining projection would be to describe a hierarchy-type machine similar to the one developed in~\cite{masur_m:curve_complexII} for the case of mapping class groups, yielding an understanding of distances in $\Out$.

\bibliographystyle{alpha}
\bibliography{outfn}

 \end{document}